\documentclass[preprint,10pt]{elsarticle_2}
\usepackage{stmaryrd}

\usepackage{amsthm,amsmath,amssymb}

\usepackage{amscd}
\usepackage{mathptmx}  

\usepackage{enumerate}

\usepackage[colorlinks,citecolor=blue,urlcolor=blue]{hyperref}  

\usepackage{graphicx}

\usepackage{latexsym}

\usepackage{verbatim}
\usepackage{pbox}

\usepackage{bussproofs}

\usepackage{tikz}
\usetikzlibrary{tikzmark}

\usepackage{bpextra}

{\usepackage{txfonts}  
   \DeclareSymbolFont{symbolsC}{U}{txsyc}{m}{n}
   \DeclareMathSymbol{\strictif}{\mathrel}{symbolsC}{74}
   \DeclareMathSymbol{\boxright}{\mathrel}{symbolsC}{128}
}

\usetikzlibrary{arrows,decorations.pathmorphing,backgrounds,positioning,fit}

\newtheorem{df}{Definition}[section]  
\newtheorem{teo}[df]{Theorem}
\newtheorem{coro}[df]{Corollary}
\newtheorem{lm}[df]{Lemma}

\newtheorem{prop}[df]{Proposition}
\newtheorem{exm}[df]{Example}

\newtheorem{prp}[df]{Proposition}
\newtheorem{fct}[df]{Fact}

\renewenvironment{proof}{\paragraph{Proof.} }{\hfill $\square$ \\}

\newenvironment{lateproof}[1] {\paragraph{{\rm\textbf{Proof of {#1}\quad}}}}{\hfill$\square$ \\}

\usepackage{bbm}

\usepackage{array}
\newcolumntype{C}[1]{>{\centering\arraybackslash}p{#1}}
\newcolumntype{L}[1]{>{\arraybackslash}p{#1}}
\usepackage{multirow}
\usepackage{float}

\newcommand{\dep}[2]{=\hspace{-3pt}({#1};{#2})}
\newcommand{\depc}[1]{=\hspace{-3pt}({#1})}
\newcommand{\con}[1]{=\hspace{-3pt}({#1})}

\newcommand{\ded}{\vdash_\sigma\xspace}

\newcommand{\cf}{\boxright}

\newcommand{\dblsetminus}{\mathbin{{\setminus}\mspace{-5mu}{\setminus}}}

\newcommand{\COV}{\ensuremath{\mathcal{CO}_{\footnotesize\mathsmaller{\dblsetminus\hspace{-0.23ex}/}}[\sigma]\xspace}}
\newcommand{\COv}{\ensuremath{\mathcal{CO}_{\footnotesize{\setminus}\mspace{-5mu}{\setminus}\hspace{-0.23ex}/}\xspace}}
\newcommand{\COd}{\ensuremath{\mathcal{COD}\xspace}}
\newcommand{\COD}{\ensuremath{\mathcal{COD}[\sigma]\xspace}}

\newcommand{\CO}{\ensuremath{\mathcal{CO}[\sigma]\xspace}}
\newcommand{\Co}{\ensuremath{\mathcal{CO}\xspace}}

\newcommand{\F}{\mathcal{F}}
\newcommand{\G}{\mathcal{G}}
\newcommand{\I}{\mathcal{I}}
\newcommand{\K}{\mathcal{K}}

\newcommand{\End}{\mathrm{En}}
\newcommand{\Exo}{\mathrm{Ex}}
\newcommand{\Con}{\mathrm{Cn}}

\newcommand{\Dom}{\mathrm{Dom}}
\newcommand{\Ran}{\mathrm{Ran}}
\newcommand{\ASS}{\ensuremath{\mathbb{A}_\sigma}}
\newcommand{\FUN}{\ensuremath{\mathbb{F}_\sigma}}
\newcommand{\FUNr}{\ensuremath{\mathbb{F}^0_\sigma}}
\newcommand{\SEM}{\ensuremath{\mathbb{S}\mathbbm{em}_\sigma}}
\newcommand{\GCT}{\ensuremath{\mathbb{G}_\sigma}}
\newcommand{\CT}{\ensuremath{\mathbb{C}_\sigma}}

\newcommand{\SET}[1]{\mathbf{#1}}

\usepackage{relsize}

\usepackage{xspace}

\newcommand{\vvee}{\raisebox{1pt}{\ensuremath{\mathop{\,\mathsmaller{\dblsetminus\hspace{-0.23ex}/}\,}}}}
\newcommand{\bigvvee}{\ensuremath{\mathop{\mathlarger{\mathlarger{\mathbin{{\setminus}\mspace{-5mu}{\setminus}}\hspace{-0.33ex}/}}}}\xspace}

\newcommand{\todoy}[1]{}
\newcommand{\todob}[1]{}

\newcommand{\corrb}[1]{#1}

\theoremstyle{definition}

\theoremstyle{remark}

\journal{
}

\begin{document}

\begin{frontmatter}

 \title{Characterizing counterfactuals and dependencies over causal and generalized causal teams}

  \author{Fausto Barbero
  }
  \address{Department of Philosophy, History and Art Studies,  University of Helsinki\\
    PL 24 (Unioninkatu 40),
    00014 University of Helsinki, Finland\\
    fausto.barbero@helsinki.fi\\
    }%
  \author{Fan Yang}
  \address{Department of Mathematics and Statistics, University of Helsinki\\
    PL 68 (Pietari Kalmin katu 5),
    00014 University of Helsinki, Finland\\
    fan.yang.c@gmail.com\\
    }


\begin{abstract}
 We analyze the  causal-observational languages that were introduced in Barbero and Sandu (2018), which allow discussing interventionist counterfactuals and functional dependencies in a unified framework.
In particular, we systematically investigate the expressive power of these languages in causal team semantics, and we provide complete natural deduction calculi for each language. Furthermore, we introduce a generalized semantics which allows representing uncertainty about the causal laws, and analyze the expressive power and proof theory of the causal-observational languages over this enriched semantics. 
\end{abstract}



\end{frontmatter}



\section{Introduction}

Counterfactual conditionals express the modality of \emph{irreality}: they describe what \emph{would} or \emph{might} be the case in circumstances that diverge from the actual state of affairs. Pinning down the exact meaning and logic of counterfactual statements has been the subject of a large literature (see e.g. \cite{Sta2019} for a survey). We are interested here in a special case: the \emph{interventionist} counterfactuals, which emerged from the literature on causal inference (\cite{SpiGlySch1993,Pea2000,Hit2001,Woo2003}). Under this reading, a conditional $\SET X=\SET x\cf\psi$ expresses that $\psi$ would hold if we were to intervene on the given system, by subtracting the variables $\SET X$ to their current causal mechanisms and forcing them to take the values $\SET x$. %

The \emph{logic} of interventionist counterfactuals has been mainly studied in the semantic context of \emph{deterministic causal models} (\cite{GalPea1998,Hal2000,Bri2012,Zha2013}), which consist of an assignment of values to variables together with a system of \emph{structural equations} that describe the causal connections among variables. 
In \cite{BarSan2018}, causal models were generalized to \emph{causal teams}, in the spirit of \emph{team semantics} (\cite{Hod1997,Vaa2007}), by considering a \emph{set} of assignments (a ``{\em team}'')  instead of a single assignment. This 
opens the possibility of describing e.g. \emph{uncertainty}, \emph{observations}, and \emph{dependencies} (\cite{Vaa2007,Eng2012,GraVaa2013,Gal2012,Vaa2008,
YanVaa2016,GraHoe2018,BarSan2018,BarSan2019}).
In this paper we also consider a further generalization of causal team semantics. While causal teams only allow describing uncertainty about the current state of the variables (the assignment), their generalized version also encode uncertainty about the causal laws.


One of the main reasons for introducing causal teams was the possibility of comparing,  within a unified semantic framework, the logic of dependencies of causal nature (those definable in terms of interventionist counterfactuals) against that of contingent dependencies (those that occur e.g. in sets of empirical data, and which have been thoroughly studied in the literature on team semantics).
 In the context of team semantics the most commonly studied type of contingent dependency  is functional determinacy, which is represented by the so-called \emph{dependence atoms} $\dep{\SET X}{Y}$, whose intended meaning is that the values of the variables in $\SET X$ determine the value of $Y$. More precisely, this means that assignments in the team 
 which agree over $\SET X$ also agree over $Y$. On the causal side, there is a large variety of notions of causal dependence; see for example \cite{Woo2003} for a taxonomy of the most common ones. A notable example is the notion of \emph{direct cause}. A variable $X$ is a direct cause of another variable  $Y$ if there is a way of manipulating $X$, while holding fixed all variables distinct from $Y$, so that the value of $Y$ will change. This can be made precise and expressed as a boolean combination of counterfactuals: 
\[
\bigvee_{\SET z,x\neq x',y \neq y'} \big( (\SET Z = \SET z \land X=x)\cf Y=y \land  (\SET Z = \SET z \land X=x')\cf Y=y' \big)
\]
where $\SET Z$ is the set of all variables distinct from $X$ and $Y$, and the disjunction ranges over the possible values of $\SET z,x,x',y,y'$ (with $x\neq x'$, $y\neq y'$).
\cite{BarSan2019} provided a complete axiomatization for a language $\Co$ which describes causal dependencies (but no  contingent dependencies)\footnote{More precisely, what was axiomatized was a class of languages $\CO$, each determined by a \emph{signature} $\sigma$. The signatures are as in \cite{Bri2012}.}; the resulting logic is a fragment of the system given by Briggs in \cite{Bri2012}. More generally, \cite{BarSan2018} and \cite{BarSan2019} give anecdotal evidence of the interactions between the two types of dependence, but offer no complete axiomatizations for languages that also involve contingent dependencies. 
In this paper we fill this gap in the literature by providing complete deduction systems (in natural deduction style) for the languages $\COd$ and $\COv$ (from \cite{BarSan2018}), which 
enrich the basic counterfactual language, respectively, with atoms of functional dependence $\dep{\SET X}{Y}$, 
or with the global disjunction $\vvee$ (in terms of which functional dependence is definable).  We also give semantical characterizations, for $\COd$, $\COv$ and the  basic counterfactual language $\Co$, in terms of definability of classes of causal teams.

All these results will also be extended to the generalized causal team semantics. We point out that the methods developed in the literature (\cite{CiaRoe2011,YanVaa2016}) for proving completeness results in team semantics can be smoothly adapted to this generalized semantics, whereas these methods do not directly apply to the basic causal team semantics. The technical hindrance here is the fact that it is not always possible to define reasonably the \emph{union} of two causal teams; this leads to the failure of the disjunction property for $\vvee$, which is a crucial element for the methods of \cite{CiaRoe2011,YanVaa2016}. In order to obtain, nonetheless, completeness results over causal team semantics, our strategy is to add to the systems for the generalized semantics further rules or axioms that characterize the property of being a causal team (i.e. of encoding \emph{certainty} about the causal connections), or more precisely, the property of being a \emph{uniform team} (as the languages we consider cannot tell apart causal teams from uniform teams).




The paper is organized as follows.
Section \ref{SECSYNSEM} introduces the formal languages and two versions of causal team semantics. Section \ref{SECCHARFUN} deepens the discussion of the functions which describe causal mechanisms, addressing issues of definability and the treatment of dummy arguments. 
Section \ref{SECCODED} characterizes semantically the language $\Co$ and reformulates in natural deduction form the $\Co$ calculi that come from \cite{BarSan2019}.  
An interesting simple consequence of the semantical analysis is a proof that the counterfactual is not definable in terms of the remaining operators.  
Section \ref{SECCOMPLETE} gives semantical characterizations for $\COd$ and $\COv$, and complete natural deduction calculi for both types of semantics. We conclude the paper in Section \ref{sec:conclusion}. An early version of this paper has already appeared as \cite{BarYan2020} in conference proceedings. 


		


\section{Syntax and semantics} \label{SECSYNSEM}

\subsection{Formal languages}\label{SECLANG}
In this paper, we  consider three classes of causal languages, which were originally introduced in \cite{BarSan2018}. 
 Each of the 
 causal languages is parametrized by a \textbf{signature} $\sigma$, which 
 is a pair $\sigma=(\Dom,\Ran)$ with $\Dom$ being a nonempty finite set of \textbf{variables}, and $\Ran$  a function that associates to each variable $X\in \Dom$ a nonempty finite set $\Ran(X)$ (called the \textbf{range} of $X$) of \textbf{constant symbols} or \textbf{values} $x$. 
 \footnote{Note that we do not encode a distinction between exogenous and endogenous variables into the signatures, as done in \cite{Hal2000}. Instead, we follow the style of Briggs \cite{Bri2012}. Doing so will result in more general completeness results.}
Throughout the paper we use the Greek letter $\sigma$ to denote an arbitrary signature. 


A boldface capital letter $\SET X$ stands for a sequence $\langle X_1,\dots,X_n\rangle$ of variables; similarly, a boldface lower case letter $\SET x$ stands for a sequence $\langle x_1,\dots,x_n\rangle$ of values.  We will sometimes abuse notation and treat $\SET X$ and $\SET x$ also as sets. We abbreviate the conjunction of equalities 
$X_1=x_1 \land \dots \land X_n=x_n$ as $\SET X = \SET x$, where we always assume that $x_i\in Ran(X_i)$, and also call $\SET X = \SET x$ an equality 
(over $\sigma$). We write $\Ran(\SET X)$ for $\Ran(X_1)\times \cdots\times \Ran(X_n)$. 

Given a signature $\sigma=(\Dom,\Ran)$, formulas of our basic languages $\CO$ are formed according to the grammar:
$$\alpha::= X=x \mid   \neg \alpha   \mid  \alpha\land\alpha   \mid  \alpha\lor\alpha   \mid  
 \SET X=\SET x\cf\alpha$$
where $\SET X \cup \{X\}\subseteq \Dom$, $x\in \Ran(X)$, $\SET x \in \Ran(\SET X)$.\footnote{We are identifying syntactical variables and values with their semantical counterpart, following the conventions of the literature on interventionist counterfactuals, e.g. \cite{GalPea1998,Hal2000,Bri2012,Zha2013}. In this convention, distinct symbols (e.g. $x,x'$) denote distinct objects. Obviously, syntax and semantics could in principle be also kept completely separated; see \cite{BarSan2019} for details.} 
 The connective $\cf$ is used to form \emph{interventionist counterfactuals}, a peculiar type of conditional formula. Following the literature on team semantics, we also call the disjunction $\lor$ \emph{tensor disjunction}. 
Throughout the paper, we 
reserve the first letters of the Greek alphabet, $\alpha$ and $\beta$ for $\CO$ formulas.

We also consider two extensions of $\CO$, obtained by adding the \emph{global disjunction} $\vvee$, 
or dependence atoms $\dep{\SET X}{Y}$: 
\begin{itemize}
\item[-] $ \COV :~~ \varphi::= X=x \mid   \neg \alpha   \mid  \varphi\land\varphi   \mid \varphi\lor\varphi \mid \varphi\vvee\varphi  \mid 
\SET X=\SET x\cf\varphi $

\item[-] $\COD :~~ \varphi::= X=x \mid \ \dep{\SET X}{Y} \mid  \neg \alpha   \mid \varphi\land\varphi \mid \varphi\lor\varphi   \mid  
 \SET X=\SET x\cf\varphi$

\end{itemize}
where $\SET X \cup \{X,Y\}\subseteq \Dom$, $x\in \Ran(X)$, $\SET x \in \Ran(\SET X)$,  $\alpha\in\CO$.   According to our syntax, negation is allowed to occur only in front of $\CO$-formulas $\alpha$.\footnote{\corrb{Following the tradition in the field of logics of (in)dependence, we do not allow the negation to occur in front of formulas with occurrences of $\vvee$ or of dependence atoms. The main reason is that, if we apply to such formulas the usual semantic clause for negation, we obtain many trivializing results such as $\neg \dep{\SET X}{Y} \equiv \bot$ and $\neg (\psi \vvee \chi) \equiv \neg (\psi \lor \chi)$\todoy{I think this second one is not that problematic, perhaps...}.} 
} This syntax for negation is more liberal than that in \cite{BarSan2018}, where formulas are assumed to be always in negation normal form. 

%



\subsection{Team semantics over causal teams} \label{SUBSCT}

We study two variants of team semantics for the three causal languages. In this section, we review (with some modifications) the causal team semantics from \cite{BarSan2019}. 




An \textbf{assignment} over  a signature $\sigma=(\Dom,\Ran)$ 
is a map
$s:\Dom\rightarrow\bigcup_{X\in \Dom}\Ran(X)$ such that $s(X)\in \Ran(X)$
for every $X\in \Dom$. 
Denote by $\ASS$ the set of all assignments over $\sigma$. 
A (non-causal) \textbf{team} $T^-$ over $\sigma$ 
is a set of assignments over $\sigma$, i.e., $T^-\subseteq \ASS$. The reason for our choice of the notation $T^-$ instead of $T$ will become clear later in this section. 


\begin{exm}\label{exm:team}
Let $\Dom = \{U,X,Y,Z\}$ be a set of variables, and let 
\[\Ran(U)=\Ran(X)=\{0,1\}, ~\Ran(Y)=\{1,2\}, \text{ and }~\Ran(Z)=\{2,3,4,5,6\}.\] 
Consider  the team $T^-=\{s_1,s_2\}$ over the signature $\sigma = (\Dom,\Ran)$, where 
\[s_1= \{(U,0),(X,0),(Y,1),(Z,2)\}\text{ and }s_2= \{(U,1),(X,1),(Y,2),(Z,6)\}.\]
The team $T^-$ clearly respects the ranges of the variables, as, e.g., $s_2(Z)=6\in \Ran(Z)$. We can  represent the team $T^-$ as a table:
\begin{center}
\begin{tabular}{|c|c|c|c|}
\hline
 \multicolumn{4}{|l|}{U \ \ X \ \ \ Y \ \ \ Z} \\
\hline
 $0$ & $0$ & $1$ & $2$ \\
\hline
 $1$ & $1$ & $2$ & $6$ \\
\hline
\end{tabular}
\end{center}
The first row of the table represents the assignment $s_1$ and the second row represents $s_2$ in the obvious manner. For instance, the table cell with the number ``6'' represents the fact that $s_2(Z)=6$, and so on.
%
%
%
\end{exm}


A \textbf{system of functions} $\F$ over $\sigma$ 
is a function  that assigns to each variable $V$ in a domain $\End(\F)\subseteq \Dom$ a set $PA^{\F}_V\subseteq \Dom\setminus\{V\}$ of  \textbf{parents} of $V$, and a function  $\F_V: \Ran(PA^{\F}_V)\rightarrow \Ran(V)$.\footnote{We  identify the set $PA^{\F}_V$ with a sequence, in a fixed lexicographical ordering.} Variables in the set $\End(\F)$ are called \textbf{endogenous variables}  of $\F$, and variables in $\Exo(\F)=\Dom\setminus \End(\F)$ are called \textbf{exogenous variables} of $\F$. 

Denote by $\FUN$ the set of all systems of functions over $\sigma$, which is clearly finite. 
Any system of functions $\F\in \FUN$ can be naturally associated with a (directed) graph $G_{\F}=(\Dom,E_{\F})$, defined as 
$$(X,Y)\in E_\F\text{ iff }X\in PA_Y^\F.$$
A \textbf{cycle} is a subset of $E_\F$ of the form $\{(X_1,X_2),(X_2,X_3),\dots,(X_n,X_1)\}$, where $X_1,\dots X_n$ ($n\geq 1$) are distinct variables. 
We say that  $\F$ is \textbf{recursive} if $G_\F$ is acyclic, i.e., 
no subset of $E_{\F}$ is a cycle. 

\begin{exm}\label{exm:system_of_functions}
Consider the system of functions $\F$ (over the signature $\sigma$ from Example \ref{exm:team})  
described by the following graph $G_\F$ and system of equations: 
\begin{center}
\begin{tabular}{cccc}
 \multicolumn{4}{l}{ } \\ 
 \multicolumn{4}{l}{
U\tikzmark{FROMUa}  \, \tikzmark{TOXa}X\tikzmark{FROMXa} \,  \tikzmark{TOYa}Y\tikzmark{FROMYa}   \, \   \tikzmark{TOZa}Z} \\
\\
\end{tabular}
 \begin{tikzpicture}[overlay, remember picture, yshift=.25\baselineskip, shorten >=.5pt, shorten <=.5pt]
	\draw [->] ([yshift=3pt]{pic cs:FROMUa})  [line width=0.2mm] to ([yshift=3pt]{pic cs:TOXa});
	\draw [->] ([yshift=3pt]{pic cs:FROMXa})  [line width=0.2mm] to ([yshift=3pt]{pic cs:TOYa});
	\draw [->] ([yshift=3pt]{pic cs:FROMYa})  [line width=0.2mm] to ([yshift=3pt]{pic cs:TOZa});
  
	\draw ([yshift=7pt]{pic cs:FROMUa})  edge[line width=0.2mm, out=30,in=160,->] ([yshift=9pt]{pic cs:TOZa});
	\draw ([yshift=8pt]{pic cs:FROMXa})  edge[line width=0.2mm, out=20,in=160,->] ([yshift=6pt]{pic cs:TOZa});
  \end{tikzpicture}
	\hspace{30pt} 
	$\left\{
	\begin{array}{lcl}
	\F_X(U) & = & U \\
	\F_Y(X) & = & X+1 \\
	\F_Z(X,Y,U) & = & 2Y+X+U \\
	\end{array}
	\right.$
\end{center}
The system of equations above describes the ``laws'' that generate the values of the variables $X,Y$ and $Z$ in terms of the values of  other variables. For instance, in the graph $G_{\F}$ the arrow from $U$ to $Z$ represents an edge $(U,Z)$. The equations are consistent with the graph $G_{\F}$. 
Since $G_\F$ contains no cycles, $\F$ is recursive. The variable $U$ with no incoming arrows is the only exogenous variable; the other variables are endogenous. That is, $\Exo(\F) = \{U\}$, and $\End(\F)= \{X,Y,Z\}$. The parents of each endogenous variable can also be easily read off from the above graph and equations: $PA^\F_X =\{U\}$, $PA^\F_Y =\{X\}$ and $PA^\F_Z =\{X,Y,U\}$. 
%
\end{exm}

We say that an assignment $s\in \ASS$ is \textbf{compatible} with a system  of functions $\mathcal F\in \FUN$ if for all endogenous variables $V\in \End(\F)$, 
$$s(V)=\F_V(s(PA_V^{\F})).$$





\begin{df}
A \textbf{causal team} over a signature $\sigma$  is a pair $T = (T^-,\F)$ consisting of \begin{itemize}
\item[-] a team $T^-$ over $\sigma$, called the \textbf{team component} of $T$, and
\item[-] a system of functions $\F$ over $\sigma$, called the \textbf{function component} of $T$,
\end{itemize}
where all assignments $s\in T^-$ are compatible with the function component $\F$.
\end{df}

\noindent In particular, a pair $T=(\emptyset,\F)$ with the team component being the empty set $\emptyset$ is also a causal team. For simplicity, over the same signature $\sigma$, we identify all teams with the empty team component and different function components, namely we stipulate $\emptyset:=(\emptyset,\F)=(\emptyset,\G)$ for any function components $\F,\G$ over $\sigma$. We call the causal team $\emptyset$ the {\em empty causal team} (over $\sigma$). 

The graph of a causal team $T$ is defined as the graph associated to its function component; we denote it as $G_T$.
 We call $T$ \textbf{recursive} if  $G_{T}$ is acyclic. We stipulate that the empty causal team is recursive.  Throughout this paper,  we will only consider recursive causal teams.
The reasons behind this choice 
are many. First of all, the recursive case \corrb{might be} of greater philosophical importance, as it is \corrb{sometimes claimed} 
that only the recursive systems of functions admit a causal interpretation (see e.g. \cite{StrWol1960} \corrb{for a discussion}). Secondly, the logic of counterfactuals in the general case is much more complex than in the recursive case (see \cite{Hal2000}) and it requires 
a separate treatment.\footnote{In particular, it naturally involves the treatment of \emph{might}-counterfactuals, which are not downward closed formulas in the sense explained in Theorem \ref{TEOGENDW}. 
They are thus out of the reach of the methods used in the present paper.}  


Intuitively, a causal team $T$  may be seen as representing an assumption concerning the causal relationships among the variables in $\Dom$ (as encoded in $\F$) together with a range of hypotheses concerning the actual state of the system (as encoded in $T^-$). 

\begin{exm}\label{EXCT}
We now combine the (non-causal) team  $T^-$ from Example \ref{exm:team} and the system $\F$ of functions from Example \ref{exm:system_of_functions} into a causal team $T=(T^-,\F)$, as illustrated in the following diagram:
\begin{center}
$T^-$: \begin{tabular}{|c|c|c|c|}
\hline
 \multicolumn{4}{|l|}{ } \\ 
 \multicolumn{4}{|l|}{
U\tikzmark{FROMU}  \, \tikzmark{TOX}X\tikzmark{FROMX} \,  \tikzmark{TOY}Y\tikzmark{FROMY}   \, \   \tikzmark{TOZ}Z} \\
\hline
 $0$ & $0$ & $1$ & $2$ \\
\hline
 $1$ & $1$ & $2$ & $6$ \\
\hline
\end{tabular}
 \begin{tikzpicture}[overlay, remember picture, yshift=.25\baselineskip, shorten >=.5pt, shorten <=.5pt]
	\draw [->] ([yshift=3pt]{pic cs:FROMU})  [line width=0.2mm] to ([yshift=3pt]{pic cs:TOX});
	\draw [->] ([yshift=3pt]{pic cs:FROMX})  [line width=0.2mm] to ([yshift=3pt]{pic cs:TOY});
	\draw [->] ([yshift=3pt]{pic cs:FROMY})  [line width=0.2mm] to ([yshift=3pt]{pic cs:TOZ});
  
	\draw ([yshift=7pt]{pic cs:FROMU})  edge[line width=0.2mm, out=30,in=160,->] ([yshift=9pt]{pic cs:TOZ});
	\draw ([yshift=8pt]{pic cs:FROMX})  edge[line width=0.2mm, out=20,in=160,->] ([yshift=6pt]{pic cs:TOZ});
  \end{tikzpicture}
	\hspace{30pt} 
	$\left\{
	\begin{array}{lcl}
	\F_X(U) & = & U \\
	\F_Y(X) & = & X+1 \\
	\F_Z(X,Y,U) & = & 2Y+X+U \\
	\end{array}
	\right.$
\end{center}
It is easy to verify that all assignments in the team component $T^-$ are compatible with the function component $\F$, e.g., \[s_2(Z) = 6 = 2\cdot 2 + 1 + 1 = 2\cdot s(Y) + s(X) + s(U) = \F_Z(s(X),s(Y),s(U)).\]


\noindent The arrows in the upper part of the table represent the graph $G_{T}$ of the causal team $T$, which is defined as $G_{\F}$. As demonstrated in Example \ref{exm:system_of_functions}, the graph $G_{T}$ contains no cycles, thus the causal team $T$ is recursive.
%
\end{exm}


\begin{df} \label{DEFCSUBT}
Let $S=(S^-,\G)$ and $T=(T^-,\F)$ be causal teams over the same signature $\sigma$. We call $S$ a \textbf{causal subteam} of $T$, denoted as $S\subseteq T$, if $S^-\subseteq T^-$ 
and  $\G = \F$. We stipulate that the empty causal team $\emptyset$ over $\sigma$ is a causal subteam of any causal team $T$ over $\sigma$, i.e., $\emptyset\subseteq T$.

\end{df}

An equality  $\SET X = \SET x$ (i.e., $X_1=x_1\wedge\dots\wedge X_n=x_n)$ is said to be \textbf{inconsistent} if it contains two conjuncts  $X=x$ and $X=x'$ with distinct values $x,x'$; otherwise it is said to be \textbf{consistent}. For a sequence $\mathbf{X}=\langle X_1,\dots,X_n\rangle$ of variables and an assignment $s$, we write $s(\mathbf{X})$ for $\langle s(X_1),\dots,s(X_n)\rangle$.

\begin{df}[Intervention]\label{intervention_ct_df}
Let $T=(T^-,\mathcal F)$ be a causal team over some signature $\sigma=(\Dom,\Ran)$. Let $\SET X=\SET x\,(=X_1=x_1\wedge\dots\wedge X_n=x_n)$ be a consistent equality over $\sigma$. The \textbf{intervention} $do(\SET X= \SET x)$ on $T$ is the procedure that generates a new causal team $T_{\SET X = \SET x}=(T_{\SET X = \SET x}^-,\mathcal F_{\SET X=\SET x})$ over $\sigma$ defined as follows:
\begin{itemize}
\item[-] $\F_{{\SET X = \SET x}}$ is the restriction of $\F$ to $\End(\F)\setminus \SET X$, that is, the system of functions that assigns to each variable $V\in\End(\F)\setminus \SET X$ the set of parents $PA^\F_V$ and the function $\F_V$ as prescribed by $\F$;  

\item[-] $T_{\SET X=\SET x}^-=\{s^\F_{\SET X=\SET x}\mid s\in T^-\}$, where each $s^\F_{\SET X=\SET x}$ 
is an assignment compatible with $\mathcal F_{{\SET X = \SET x}}$ defined (recursively) as 
\begin{center}\(s^\F_{\SET X=\SET x}(V)=\begin{cases}
x_i&\text{ if }V=X_i\text{ for some }1\leq i\leq n,\\
s(V)&\text{ if } V\in \Exo(\F)\setminus \SET X, \\ 
\F_V(s^\F_{\SET X=\SET x}(PA_V^{\F}))&\text{ if }V\in \End(\F)\setminus \SET X.
\end{cases}\)\end{center}
\end{itemize}

\noindent Every $s^\F_{\SET X=\SET x}$ defined above represents the result of the \textbf{intervention on the single assignment} $s$ with respect to $\F$. In the following, when the $\F$ is clear from the context, we sometimes simply write $s_{\SET X=\SET x}$.

We stipulate that the intervention $do(\SET X= \SET x)$ on the empty causal team $\emptyset$ generates the empty causal team itself, i.e., $\emptyset_{\SET X= \SET x}=\emptyset$.
\end{df}

\begin{exm} \label{EXCTINT}
Recall the recursive causal team $T$ in Example \ref{EXCT}. 
By applying the intervention $do(X=1)$ to $T$, we obtain a new causal team $T_{X=1}=(T_{X=1}^-,\F_{X=1})$  as follows. The function component $\F_{X = 1}$ is determined by the following system of 
equations:
\begin{center}
$\left\{
	\begin{array}{lcl}
		(\F_{X=1})_Y(X) & = & X+1 \\
	(\F_{X=1})_Z(X,Y) & = & 2Y+X+U \\
	\end{array}
	\right.$
\end{center}
The endogenous variable $X$ of the original team $T$ becomes  exogenous in the new team  $T_{X=1}$, and the equation $\F_X(U) = U$ for $X$ is now removed.
The new team component $T_{X=1}^-$ is obtained by the rewriting procedure illustrated below:
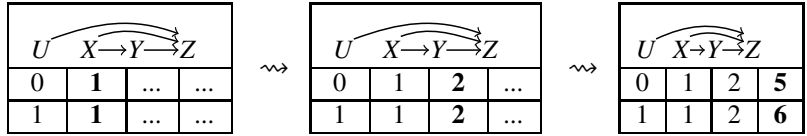
\begin{center}
 \begin{tabular}{|c|c|c|c|}
\hline
 \multicolumn{4}{|l|}{ } \\
 \multicolumn{4}{|l|}{
 \ U\tikzmark{FROMU1}  \, \ \ \tikzmark{TOX1}X\tikzmark{FROMX1} \, \ \ \tikzmark{TOY1}Y\tikzmark{FROMY1}   \, \ \ \  \tikzmark{TOZ1}Z} \\
\hline
 \ $0$ \ & \ $\mathbf{1}$ \ & ... & ... \\
\hline
 $1$ & $\mathbf{1}$ & ... & ... \\
\hline
\end{tabular}
 \begin{tikzpicture}[overlay, remember picture, yshift=.25\baselineskip, shorten >=.5pt, shorten <=.5pt]
 
	\draw [->] ([yshift=3pt]{pic cs:FROMX1})  [line width=0.2mm] to ([yshift=3pt]{pic cs:TOY1});
	\draw [->] ([yshift=3pt]{pic cs:FROMY1})  [line width=0.2mm] to ([yshift=3pt]{pic cs:TOZ1});

		\draw ([yshift=7pt]{pic cs:FROMU1})  edge[line width=0.2mm, out=20,in=160,->] ([yshift=9pt]{pic cs:TOZ1});
	\draw ([yshift=8pt]{pic cs:FROMX1})  edge[line width=0.2mm, out=20,in=165,->] ([yshift=6pt]{pic cs:TOZ1});
  \end{tikzpicture}
	\hspace{3pt} 
	$\rightsquigarrow$
	\hspace{3pt}
\begin{tabular}{|c|c|c|c|}
\hline
 \multicolumn{4}{|l|}{ } \\
 \multicolumn{4}{|l|}{
 \ U\tikzmark{FROMU2}  \, \ \ \tikzmark{TOX2}X\tikzmark{FROMX2} \, \ \ \tikzmark{TOY2}Y\tikzmark{FROMY2}   \, \ \ \  \tikzmark{TOZ2}Z} \\
\hline
 \ $0$ \ & \ $1$ \ & \ $\mathbf{2}$ \ & ... \\
\hline
 $1$ & $1$ & $\mathbf{2}$ & ... \\
\hline
\end{tabular}
 \begin{tikzpicture}[overlay, remember picture, yshift=.25\baselineskip, shorten >=.5pt, shorten <=.5pt]
 
	\draw [->] ([yshift=3pt]{pic cs:FROMX2})  [line width=0.2mm] to ([yshift=3pt]{pic cs:TOY2});
	\draw [->] ([yshift=3pt]{pic cs:FROMY2})  [line width=0.2mm] to ([yshift=3pt]{pic cs:TOZ2});

	\draw ([yshift=7pt]{pic cs:FROMU2})  edge[line width=0.2mm, out=20,in=160,->] ([yshift=9pt]{pic cs:TOZ2});
	\draw ([yshift=8pt]{pic cs:FROMX2})  edge[line width=0.2mm, out=20,in=165,->] ([yshift=6pt]{pic cs:TOZ2});

  \end{tikzpicture}
	\hspace{3pt} 
	$\rightsquigarrow$
	\hspace{3pt}
\begin{tabular}{|c|c|c|c|}
\hline
 \multicolumn{4}{|l|}{ } \\
 \multicolumn{4}{|l|}{
  U\tikzmark{FROMU3}  \,   \tikzmark{TOX3}X\tikzmark{FROMX3} \,   \tikzmark{TOY3}Y\tikzmark{FROMY3}   \,  \   \tikzmark{TOZ3}Z} \\
\hline
  $0$  &  $1$  &  $2$ & $\mathbf{5}$ \\
\hline
 $1$ & $1$ & $2$ & $\mathbf{6}$ \\
\hline
\end{tabular}
 \begin{tikzpicture}[overlay, remember picture, yshift=.25\baselineskip, shorten >=.5pt, shorten <=.5pt]
 
	\draw [->] ([yshift=3pt]{pic cs:FROMX3})  [line width=0.2mm] to ([yshift=3pt]{pic cs:TOY3});
	\draw [->] ([yshift=3pt]{pic cs:FROMY3})  [line width=0.2mm] to ([yshift=3pt]{pic cs:TOZ3});

	\draw ([yshift=7pt]{pic cs:FROMU3})  edge[line width=0.2mm, out=25,in=160,->] ([yshift=9pt]{pic cs:TOZ3});
	\draw ([yshift=8pt]{pic cs:FROMX3})  edge[line width=0.2mm, out=20,in=160,->] ([yshift=6pt]{pic cs:TOZ3});

  \end{tikzpicture}
\end{center}
In the first step,  rewrite the $X$-column with value $1$. Then, update (recursively) the other columns using the functions from $\F_{{X = 1}}$. In this step, only the columns that correspond to ``descendants'' of $X$ will be modified, and the order in which these columns should be updated is completely determined by the (acyclic) graph $G_{T_{X=1}}$ of $T_{X=1}$. Since the variable $X$ becomes exogenous after the intervention, all arrows pointing to $X$ have to be removed; in this case, the arrow from $U$ to $X$. 
We refer the reader to \cite{BarSan2019} for more details and a proof of termination for this rewriting procedure. 
\end{exm}

\begin{df}\label{ct_semantics_df}
Let $\varphi$ be a formula in the language $\COV$ or  $\COD$, and $T=(T^-,\mathcal F)$ a causal team over $\sigma$. We define the satisfaction relation $T\models^c\varphi$ (or simply $T\models\varphi$) over causal teams inductively as follows: 
\begin{itemize}
\item[-] $T\models X=x$ ~~iff~~ for all $s\in T^-$, $s(X)=x$.\footnote{Note once more that  the symbol $x$ is used as both  a syntactical and a semantical object.}
\item[-] $T\models\, \dep{\SET X}{Y}$ ~~iff~~ for all $s,s'\in T^-$, $s(\SET X)=s'(\SET X)$ implies $s(Y)=s'(Y)$.
\item[-] $T\models \neg\alpha$ ~~iff~~ for all $s\in T^-$, $(\{s\},\mathcal F)\not \models \alpha$.
\item[-] $T\models \varphi\land \psi$ ~~iff~~  $T\models \varphi$ and $T\models \psi$.
\item[-] $T\models \varphi\lor \psi$ ~~iff~~ there are two causal subteams
$T_1,T_2\subseteq T$ such that $T_1^-\cup T_2^- = T^-$, $T_1\models \varphi$ and $T_2\models \psi$. 
\item[-] $T\models\varphi\vvee\psi$ ~~iff~~ $T\models \varphi$ or $T\models\psi$.
\item[-] $T\models \SET X = \SET x \cf \varphi$ ~~iff~~ $\SET X = \SET x$ is inconsistent or $T_{\SET X=\SET x}\models\varphi$.
\end{itemize}
\end{df}


As usual, define $\bot:= X=x\wedge X\neq x$, where $X\neq x$ is short for $\neg (X=x)$. Clearly, $T\models\bot$ iff $T=\emptyset$.

We write $T\models \Gamma$ and say that ``$T$ satisfies $\Gamma$'' if $T\models \varphi$ for each $\varphi\in\Gamma$. As usual, we can define semantic entailment in terms of satisfaction. For any fixed signature $\sigma$, we write that $\Gamma\models^c\varphi$ (or simply $\Gamma\models\varphi$ whenever there is no risk of ambiguity) if every (recursive) causal team $T$ of signature $\sigma$ that satisfies $\Gamma$ also satisfies $\varphi$. We say that $\varphi$ and $\psi$ are equivalent, denoted as $\varphi\equiv^c\psi$ (or simply $\varphi\equiv\psi$) if $\varphi\models^c\psi$ and $\psi\models^c\varphi$. 


We write a dependence atom $\dep{}{Y}$ with an empty first component  as $\con{Y}$. Clearly, the semantic clause for $\con{Y}$ reduces to:
\begin{itemize}
\item[-] $T\models\, \con{Y}$ ~~iff~~  for all $s,s'\in T^-$, $s(Y)=s'(Y)$.
\end{itemize}
Intuitively, the atom $\con{Y}$ states that $Y$ has a constant value in the team and is thus called a {\em constancy atom}. 
In the context of finite signatures -- i.e. allowing only finite domains and ranges, as is done throughout this paper  -- the dependence atoms are definable in terms of the constancy atoms, because of the following equivalence:
\begin{equation}\label{dep_atm_dfnable_vvee}
\dep{\SET X}{Y}\equiv\bigvee_{\SET x\in \Ran(\SET X)}(\SET X=\SET x\,\wedge \depc{Y}).
\end{equation}
Furthermore, in $\COV$, the finitude of $\sigma$ makes the constancy atoms definable:

\begin{equation}\label{con_atm_dfnable_vvee}
\displaystyle\depc{Y}\equiv \bigvvee_{y\in \Ran(Y)}Y=y.
\end{equation}
Thus, by (\ref{dep_atm_dfnable_vvee}) and (\ref{con_atm_dfnable_vvee}), the dependence atoms themselves are definable in $\COV$.

Recall that counterfactuals $\SET X=\SET x\cf \varphi$ express that $\varphi$ holds after an intervention. In \cite{BarSan2018,BarSan2019}, another conditional operator $\supset $ called the \emph{selective implication} also plays an important role. The selective implication $\alpha\supset\varphi$ (always with  a $\Co$-formula $\alpha$ in the antecedent) expresses that $\varphi$ holds after observing a fact described by $\alpha$. Its semantics is defined as: 
\begin{itemize}
    \item[-] $T \models \alpha\supset \varphi  \iff T^\alpha \models \varphi$, where $T^\alpha$ is the (unique) causal subteam of $T=(T^-,\F)$ with team component $(T^\alpha)^-=\{s \in T^- \mid (\{s\},\F)\models \alpha\}$.
\end{itemize}
It is easy to verify that $\alpha\supset\varphi\equiv\neg \alpha \lor \varphi$ in the languages considered in this paper.\footnote{The equivalence holds due to the fact that the languages are \emph{downward closed} (see Theorem \ref{TEOGENDW}). In the absence of downward closure, $\supset$ is still definable by the more general equivalence $\alpha\supset\varphi \equiv \neg\alpha \lor (\alpha \land \varphi)$.} We thus treat the selective implication $\alpha\supset\varphi$ as a shorthand in our logics.
The selective implication generalizes material implication in the sense that it behaves in the same way on \emph{singleton} causal teams; i.e., 
$$
(\{s\},\F)\models \alpha\supset \varphi \iff (\{s\},\F)\not\models \alpha \text{ or } (\{s\},\F)\models \varphi.
$$

\begin{exm}
Consider the causal team $T$ and the intervention $do(X=1)$ from Examples \ref{EXCT} and \ref{EXCTINT}. Clearly, $T_{X=1}\models Y=2$, and thus $T\models X=1 \cf Y=2$. 
We also have that $T\models\dep{Y}{Z}$, while $T_{X=1}\not\models\dep{Y}{Z}$ (contingent dependencies are not in general preserved by interventions). Observe that $T\models Y\neq 2 \lor Y=2$, while $T\not\models Y\neq 2 \vvee Y=2$.
The team $T^{X=1}$ consists only of the  assignment in the second row of the table for $T$, on which $Y=2$ holds; so we have that $T\models X=1 \supset Y=2$. 

\end{exm}

\subsection{Team semantics over generalized causal teams} \label{SUBSGCT}

Causal team semantics was developed in order to support the introduction of counterfactuals and observations within the framework of team semantics. 
When we look at one of the main interpretations of teams -- as a representation of \emph{uncertainty} among possible states of affairs -- we see an asymmetry: causal teams encode uncertainty about the state of the variables, and certainty about the causal laws. Yet, in all sciences the knowledge of the laws of causation may be as uncertain -- if not more uncertain -- than the knowledge of factual matters. \cite{BarSan2018} proposed to model uncertainty about causal laws by using the more general \emph{partially defined causal teams}. In this section we introduce a more general and at the same time simpler approach: the \emph{generalized causal teams}.


Given a signature $\sigma$, we write 
$$
\SEM:=\{(s,\mathcal F)\in \ASS\times \FUN\mid s\text{ is compatible with }\mathcal F\}.
$$
The pairs $(s,\mathcal F)\in \SEM$ can be easily identified with the \emph{deterministic causal models} (also known as \emph{deterministic structural equation models}) that are considered in the literature on causal inference (\cite{SpiGlySch1993},\cite{Pea2000}, 
etc.).
Even though they are distinct mathematical objects, it is natural to identify a causal team $T=(T^-,\mathcal F)$ with the set
\[
T^g=\{(s,\mathcal F)\mid s\in T^-\}
\]
of deterministic causal models with a common 
function component $\mathcal F$. 
We now introduce the notion of {\em generalized causal team}, in which the function component $\mathcal F$ does not have to be constant throughout the team.

\begin{df} \label{DEFSGCT}
A \textbf{generalized causal team} $T$ over a signature $\sigma$ is a set of pairs $(s,\mathcal F)\in \SEM$, that is, $T\subseteq \SEM$.
We call the set 
\[T^- := \{s \mid (s,\mathcal F) \in T \text{ for some } \mathcal F\}\]  
the \textbf{team component} of $T$.
%
%
A \textbf{causal subteam} of $T$ is a subset $S$ of $T$, denoted as $S\subseteq T$. 
In particular, the empty set $\emptyset$ is a generalized causal team, and $\emptyset\subseteq T$ for all generalized causal teams $T$.


The \textbf{union} $S\cup T$ of two generalized causal teams $S,T$ is their set-theoretic union. 
\end{df}

 Intuitively,  a generalized causal team encodes uncertainty about  which causal model governs the variables in $\Dom$ - i.e., uncertainty both on the values of the variables and on the laws that determine them. Distinct elements $(s,\mathcal F), (t,\mathcal G)$ of the same generalized causal team may also disagree on what is the set of endogenous variables, or on whether the system is recursive or not. 
 A generalized causal team is said to be \textbf{recursive} if, for each  pair $(s,\mathcal F)$ in the team, the associated graph $G_\F$ is recursive. In this paper we  only consider recursive generalized causal teams. 

\begin{exm}\label{EXGCT}
Let $\sigma=(\Dom,\Ran)$ be a signature with variables $\Dom=\{X,Y,Z\}$ and ranges  $\Ran(X)=\{1,2\}$, $\Ran(Y)=\{2,3\}$ and $\Ran(Z)=\{3,4,5\}$. 
Let $\F,\G\in \FUN$ be two (recursive) systems of functions described by the following two systems of equations:
\begin{center}
	$\left\{
	\begin{array}{lcl}
	\F_Z(X)  =  2X \\
	\end{array}
	\right.$
	\quad\quad \quad $\left\{
	\begin{array}{lcl}
	\G_Z(X,Y) = X+Y \\
	\end{array}
	\right.$
\end{center}
Consider two assignments 
\(s = \{ (X,2),(Y,2),(Z,4) \}\text{ and }t = \{ (X,1),(Y,3),(Z,4) \}\)
over $\sigma$. Now, the generalized causal team $T = \{ (s,\F), (s,\G),(t,\G)\}$ over $\sigma$ can be represented as the following table:
\begin{center}
\begin{tabular}{|c|c|c|c|}
\hline
\multicolumn{4}{|l|}{$X$  ~ $Y$ ~~ $Z$} \\
\hline
2 & 2 & 4 & $\F$ \\
\hline
2 & 2 & 4 & $\G$ \\
\hline
1 & 3 & 4 & $\G$ \\
\hline
\end{tabular}
\end{center}
The first element $(s,\F)$ differs from the second one $(s,\G)$ in its function component (as $\F\neq\G$). Since both  $\F$ and $\G$ are recursive, the generalized causal team $T$ is recursive.
%
%
%
%
\end{exm}



The above is a toy example that illustrates the content of a generalized causal team and its graphical presentation. Let us now discuss some possible applications of generalized causal teams as representation tools. Consider a situation in which we have observational data about some phenomenon. It is very common to organize the data in the form of a database or a table. When we are uncertain among alternative causal explanations $\F_1,\dots,\F_n$ that might have produced these data, we may represent the situation by a generalized causal team that contains causal models $(s,\F_1),\dots,(s,\F_n)$ for each distinct entry $s$ of the original database.\footnote{If one wants to make considerations of probabilistic nature, it might be necessary to allow multiple copies of the same causal model to occur in the model. We are not pursuing the probabilistic view in this paper, but see \cite{BarSan2018} for some ideas in this direction.} 
In the worst case, we may be completely clueless about the causal explanations, and thus our generalized causal team will include a causal model $(s,\F)$ for each system of functions $\F$ compatible with \emph{all} entries from the original database. More commonly, we are not completely ignorant about the causal explanations. Let us consider three possible types of situations of this kind.

In the first case, we know what causes what, and what does not -- in other words, we are certain about what the causal graph $G$ is -- but we have no quantitative description of the laws. 
This can be represented in a generalized causal team by including  copies of the entries for each system of functions that agrees with the causal dependencies encoded in $G$.\footnote{In a discussion with the first author, Johan van Benthem suggested the possibility of using teams to represent uncertainty about the correct causal graph. With a little extra toil, also this can be represented by a generalized causal team.} In the second case, the data arises from the observation of different individuals that may be of different types, and are thus bound to react differently to a possible intervention. 
Our uncertainty about these types and how the individual will be affected by  interventions may be represented by associating multiple function components to an individual. In other words, we can use the generalized causal teams to represent the \emph{dispositions} of objects, such as the property of being flammable (in some circumstances) or not. In the third case, we may use  generalized causal teams to represent our \emph{partial} knowledge about a causal law. This corresponds to the scenarios in which we can predict reasonably the behaviour of a physical system only within certain ranges for the values of the variables. For example, we know that if we heat the gas in a balloon, its volume will increase linearly (by Charles' law), but not if we raise its temperature by $5000 ^{\circ}$C (the balloon will melt at this temperature). 
The generalized causal team for this situation will include every system of (total) functions that agree with the known partial law within the ``safe'' range of values.\footnote{\cite{BarSan2018} proposed a different solution to this representation problem, the \emph{partially defined causal teams}. This framework can then be identified with a very special case of generalized causal team semantics. On the other hand, the kind of representation proposed in \cite{BarSan2018} might be more efficient for this specialized task.} 

Let us now come back to the formal theory. We have already mentioned that a causal team $T=(T^-,\F)$ can be identified with the generalized causal team $T^g=\{(s,\F)\mid s\in T^-\}$, which has a constant function component  in all its elements. 
Conversely, if $T$ is a nonempty 
generalized causal team in which all elements  have the same function component $\mathcal F$, i.e., $T=\{(s,\mathcal F)\mid s\in T^-\}$, 
we can  
naturally identify $T$ with the causal team 
\begin{center}\(T^c=(T^-,\mathcal F).\)\end{center}
We stipulate that the empty generalized causal team $\emptyset$ corresponds to the empty causal team $\emptyset$, i.e., $\emptyset^c=\emptyset$.
In particular, a singleton generalized causal team $\{(s,\mathcal F)\}$ corresponds to  a singleton causal team $(\{s\},\F)$. Applying a (consistent) intervention $do(\SET X = \SET x)$ on the causal team $(\{s\},\mathcal F)$ generates a causal team $(\{s^\F_{\SET X = \SET x}\},\mathcal F_{\SET X = \SET x})$ as specified in Definition \ref{intervention_ct_df}.
We can then define the result of the intervention $do(\SET X = \SET x)$ on the \emph{generalized} causal team $\{(s,\mathcal F)\}$ to be the generalized causal team $(\{s^\F_{\SET X = \SET x}\},\mathcal F_{\SET X = \SET x})^g=\{(s^\F_{\SET X = \SET x},\mathcal F_{\SET X = \SET x})\}$.
Interventions on arbitrary generalized causal teams are defined as follows. 


\begin{df}[Intervention over generalized causal teams]
Let $T$ be a (recursive) generalized causal team, and $\SET X = \SET x$  a consistent equality over $\sigma$. The intervention $do(\SET X = \SET x)$ on $T$ generates the generalized causal team 
\[T_{\SET X = \SET x} :=\{(s^\F_{\SET X = \SET x},\F_{\SET X = \SET x})\mid (s,\F)\in T\}.\]
\end{df}

\begin{exm}
Consider  the generalized causal team  $T$ from Example \ref{EXGCT}. The intervention $do(Y=1)$ is computed analogously as it would be done on a causal team, except that in this case the entries in each row must be updated according to the 
function component corresponding to the row. The team $T_{\SET X = \SET x}$ after the intervention is as follows:
\begin{center}
\begin{tabular}{|c|c|c|c|}
\hline
\multicolumn{4}{|l|}{X  \, \ Y \ \ \ Z} \\
\hline
2 & \textbf{1} & 4 & $\F_{Y=1}$ \\
\hline
2 & \textbf{1} & \textbf{3} & $\G_{Y=1}$ \\
\hline
1 & \textbf{1} & \textbf{2} & $\G_{Y=1}$ \\
\hline
\end{tabular}
\end{center}

\noindent The function components are $\F_{Y=1}=\F$ and $\G_{Y=1}=\G$, since $Y$ is exogenous for both $\F$ and $\G$. Since $Y$ is not a parent of any endogenous variable of $\F$,
the values of other variables in the first row  remain unchanged after the intervention.

\end{exm}

%

\begin{df}
Let $\varphi$ be a formula of the language $\COV$ or  $\COD$, and let $T$ be a generalized causal team over $\sigma$. The satisfaction relation $T\models^g\varphi$ (or simply $T\models\varphi$) over generalized causal teams is defined in the same way as in Definition \ref{ct_semantics_df}, except for slight differences in the following clause:
\begin{itemize}
\item[-] $T\models^g \neg\alpha$ ~~ iff ~~ for all $(s,\mathcal F)\in T$, $\{(s,\mathcal F)\}\not \models \alpha$.
\end{itemize}
\end{df}

For any fixed $\sigma$, we also write $\Gamma\models^g\varphi$ (or simply $\Gamma\models\varphi$) if all (recursive) generalized causal teams of signature $\sigma$ that satisfy $\Gamma$ also satisfy $\varphi$; and $\varphi \equiv^g \psi$ (or simply $\varphi \equiv \psi$) if $\varphi \models^g\psi$ and $\psi\models^g\varphi$. 

In the next theorem we list some closure properties for our logics over both causal teams and generalized causal teams. The proof  is left to the reader, or  see \cite{BarSan2019} for the causal team case.


\begin{teo} \label{TEOGENDW}  
Let $T,S$ be (generalized) causal teams over some signature $\sigma$, and let $\varphi$ be a $\COV$-formula or $\COD$-formula.
\begin{description}
\item[Empty team property] $\emptyset\models\varphi$. 
\item[Downward closure] If $T\models\varphi$ and $S\subseteq T$, then $S\models\varphi$.
\end{description}

\noindent In addition, $\CO$-formulas $\alpha$ are closed under unions:
\begin{description}
\item[Union closure] If $T\models\alpha$ and $S\models\alpha$, then $T\cup S\models\alpha$, whenever $T\cup S$ is defined.\footnote{The unions of generalized causal teams are always defined, as in definition \ref{DEFSGCT}. The union of two causal teams with distinct function components, instead, will not generally produce a causal team. See section \ref{SECCOEXPR} for a definition of the union of (similar) causal teams.}
\end{description}
\noindent The empty team property, downward closure and union closure together are equivalent to the flatness property, which $\CO$-formulas satisfy: 
\begin{description}
\item[Flatness] 

\noindent\(T\models \alpha \iff  
(\{s\},\F)\models^c\alpha \text{ for all }s
\) in $T=(T^-,\F)$ 

\(\hspace{100pt}(\text{resp. } \{(s,\F)\}\models^g\alpha \text{ for all }(s,\F)\in T).\)
\end{description}
\end{teo}

The next lemma shows that the team semantics over causal teams and that over generalized causal teams with a constant function component are essentially equivalent. 
\begin{lm}\label{LEMIDENTIFY}
Let $\varphi$ be a $\COV$ or a $\COD$ formula.
\begin{enumerate}
\item[(i)] For any causal team $T$, we have that $T\models^c\varphi\iff T^g\models^g\varphi.$
\item[(ii)] For any 
generalized causal team $T$ with a unique function component, we have that $T\models^g\varphi\iff T^c\models^c\varphi.$  
\end{enumerate}
\end{lm}
 \begin{proof}
Both items are proved by straightforward induction. 
 \end{proof}

\begin{coro}\label{gct_sem_cons_2_ct}
For any set $\Delta\cup\{\alpha\}$  of $\CO$-formulas, $\Delta\models^{g}\alpha$ iff $\Delta\models^{c}\alpha$. 
\end{coro}
\begin{proof}

Assume $\Delta\models^{c}\alpha$, and let $T$ be a generalized causal team such that $T\models^g\Delta$. If $T=\emptyset$, by the empty team property we have $T\models \alpha$, and we are done. Otherwise observe that, by the flatness of $\Delta$,  $\{(s,\F)\}\models^g \Delta$ for each $(s,\F)\in T$. By Lemma \ref{LEMIDENTIFY}, we then have that
 $(\{s\},\mathcal F)\models^{c}\Delta$ for each  $(s,\F)\in T$. By the assumption $\Delta\models^{c}\alpha$,  we have $(\{s\},\mathcal F)\models^{c}\alpha$ for each $(s,\F)\in T$. By lemma \ref{LEMIDENTIFY} again, $\{(s,\mathcal F)\}\models^{g}\alpha$ for each $(s,\F)\in T$. By flatness of $\alpha$, we conclude $T\models^g\alpha$.
The  converse direction is proved analogously. 
\end{proof}

\section{Characterizing function components}\label{SECCHARFUN}

As is common in the literature on causation, our formal languages only talk of variables and values; they do not explicitly mention the causal laws nor the functions that describe them. In order to understand the expressive power of these languages, we  need to first understand to what extent they allow an implicit description of the laws. In this section we will answer to this question; the answer is: to a great extent. In the following subsection we begin addressing the aspects of the functional laws that can\emph{not} be described by our languages. 


\subsection{Equivalence of function components 
} \label{SUBSEQUIV}


Consider a binary function $f$ and a quaternary function $g$ defined as \(f(X,Y)=X+Y\text{ and }g(X,Y,Z_1,Z_2)=X+Z_1+0/Z_2+(Y-Z_1).\) 
Essentially $f$ and $g$ are the same function: $Z_1,Z_2$ 
are dummy arguments of $g$. 
In some formal presentations of causal reasoning, the distinction between $f$ and $g$ is quotiented out. We preferred to maintain this distinction in the present paper, following the approaches in \cite{Bri2012,BarSan2018,BarSan2019}. This type of distinction might be needed in case the framework is applied to realistic scenarios. Given a complex description of a function, it might be unfeasible to check whether some of its arguments are dummy in case the variable ranges are large. On the other hand, however, this additional degree of generality brings with it a burden of bookkeeping, which we now address. 


First of all, we define formally 
what it means for two functions (and more generally, two function components) to be equivalent up to dummy arguments. 

\begin{df}\label{DEFSIMFG}
Let $\mathcal F,\mathcal G$ be two function components over  $\sigma = (\Dom,\Ran)$.
\begin{itemize}

\item[-] Let $V\in \Dom$. The two functions $\F_V$ and $\G_V$ are said to be equivalent up to dummy arguments, denoted as $\F_V \sim \G_V$, if 
for any $\SET x\in \Ran(PA_V^{\F}\cap PA_V^{\G})$, $\SET y\in \Ran(PA_V^{\F}\setminus PA_V^{\G})$ and  $\SET z \in \Ran(PA_V^{\G}\setminus PA_V^{\mathcal F})$, 
we have that 
$$\F_V(\SET x\SET y)= \G_V(\SET x\SET z)$$
 (where we assume w.l.o.g. the shown orderings of the arguments of the functions). 

\item[-] 
Let $\Con(\F)$ denote the set of endogenous variables $V$ of $\F$ for which  $\F_V$ is a constant function, i.e., for some fixed $c\in \Ran(V)$, 
\[\F_V(\SET p)=c\text{ for all }\SET p\in PA_V^{\F}.\] 
We say that $\F$ and $\G$ are \textbf{equivalent up to dummy arguments}, denoted as $\mathcal F \sim \mathcal G$, if 
\begin{itemize}
\item 
$\End(\mathcal F)\setminus \Con(\mathcal F) = \End(\mathcal G)\setminus \Con(\mathcal G)$, 
\item 
and $\F_V \sim \G_V$ holds for all $V\in \End(\mathcal F)\setminus \Con(\mathcal F)$. 

\end{itemize}
\end{itemize}
\end{df}

\noindent It is easy to see that $\sim$ is an equivalence relation. 
Furthermore, the relation $\sim$ is  preserved under  
interventions. 

\begin{fct}\label{LEMINTSIM1}
For any function components $\mathcal F,\mathcal G\in \FUN$ and consistent equality $\SET X=\SET x$ over $\sigma$, we have that
$\mathcal F \sim \mathcal G$  implies $\mathcal F_{\SET X = \SET x} \sim \mathcal G_{\SET X = \SET x}$.
\end{fct}

\begin{proof}
Straightforward; see \cite{BarYan2020}.
\end{proof}

We now generalize the equivalence relation $\sim$ to the causal team level. 
Two nonempty causal teams $T=(T^-,\mathcal F)$ and $S=(S^-,\mathcal G)$ of the same signature $\sigma$ are said to be 
\textbf{similar}, denoted as $T\sim S$, if $\mathcal F \sim \mathcal G$; 
\textbf{equivalent}, denoted as $T\approx S$, if $T\sim S$ and $T^- = S^-$. This definition of equivalence is meant to describe, semantically, the conditions under which $S$ and $T$ cannot be distinguished by our formal languages.
%
%
%
%

Let us now define equivalence also for generalized causal teams. For any generalized causal team $T$, 
write  
\[
T^\mathcal F := \{(s,\mathcal G)\in T \mid \mathcal G\sim \mathcal F \}.
\]
Two generalized causal teams $S$ and $T$ are said to be  \textbf{equivalent}, denoted as $S\approx T$, if   $(S^{\F})^-=(T^{\F})^-$  for all $\F \in \FUN$. 

Moreover, we define the equivalence relation $\approx$  between elements of $\SEM$ by taking 
\[
(s,\F)\approx(t,\G) ~\text{ iff }~ \{(s,\F)\}\approx \{(t,\G)\}~~(\text{ iff }~s=t\text{ and }\F\sim\G).
\]

We now show that formulas in the logics we consider in this paper are invariant under causal equivalence. In other words, none of our languages can distinguish causal teams which are equivalent up to dummy arguments. 

\begin{teo}[Invariance under causal equivalence]\label{PROPEQUIV} 
Let $T,S$ be two (generalized) causal teams over $\sigma$ with $T\approx S$. For any formula $\varphi$ of 
$\CO$, $\COV$ or $\COD$, 
we have that \(T\models\varphi\iff S\models\varphi.\)
\end{teo}
\begin{proof}
The theorem is proved by induction on $\varphi$. The case $\varphi=\,\SET X=\SET x\cf \psi$ follows  from the fact that
$T_{\SET X = \SET x}\approx S_{\SET X = \SET x}$
 (Fact \ref{LEMINTSIM1}). The case $\varphi=\psi\vee\chi$ for causal teams follows directly from the induction hypothesis. We now give the proof for this case for generalized causal teams. We only prove the left to right direction (the other direction is symmetric). Suppose $T\models^g\psi\vee\chi$. Then there are $T_0,T_1\subseteq T$ such that $T=T_0\cup T_1$, $T_0\models^g\psi$ and $T_1\models^g\chi$. Consider $S_i=\big\{(s,\F)\in S\mid (s,\F)\approx(s,\G)\text{ for some }(s,\G)\in T_i\big\}$ ($i=0,1$). It is 
  not hard to see 
  that $S_i\approx T_i$ ($i=0,1$) and $S=S_0\cup S_1$. By induction hypothesis we have that $S_0\models^g\psi$ and $S_1\models^g\chi$. Hence $S\models^g\psi\vee\chi$.
\end{proof}


\subsection{Characterization of the function components} 

In this subsection we show that, up to dummy arguments, every system of functions of signature $\sigma$ is definable in $\CO$.  This result is crucial for adapting the standard methods of team semantics to the causal context.

For any function component $\mathcal F$ over  $\sigma$, define a $\CO$-formula

\begin{center}\(
\displaystyle\Phi^{\mathcal F}:= \bigwedge_{V\in \End(\mathcal F)\setminus \Con(\F)} \eta_\sigma(V) \land \bigwedge_{V\notin \End(\mathcal F)\setminus \Con(\F)  
} \xi_\sigma(V).
\)
\end{center}
where 
\begin{align*}
\eta_\sigma(V):= \bigwedge\big\{&(\SET W = \SET w \land PA_V^{\mathcal F} = \SET p)\cf V = \mathcal F_V(\SET p)\\
&\quad\quad\mid \SET W = \Dom \setminus(PA_V^\mathcal F\cup \{V\}) ,~\SET w \in \Ran(\SET W) ,~ \SET p\in \Ran(PA_V^\mathcal F)\big\}\\
 \text{and }~\xi_\sigma(V):= \bigwedge\big\{&V=v \supset (\SET W_V=\SET w \cf V=v)\\
&\quad\quad\quad\quad\quad\quad\quad\mid v\in \Ran(V),\SET W_V=\Dom\setminus\{V\}, \SET w\in \Ran(\SET W_V)\big\}.
\end{align*}

\vspace{5pt}

\noindent Intuitively, for each non-constant endogenous variable $V$ of $\mathcal F$, the formula $\eta_\sigma(V)$ specifies that all assignments in the (generalized) causal team $T$  in question behave exactly as required by the function $\mathcal F_V$. For each variable $V$ which, according to $\mathcal F$, is exogenous or generated by a constant function, the formula $\xi_\sigma(V)$ states that $V$ is not affected by interventions on 
 other variables. Notice that, if $V\in\Con(\F)$, then $V$ has a $\xi$ but not an $\eta$ clause. Our languages cannot tell apart an endogenous variable that is constant ``by law'' from an exogenous variable that happens to be (contingently) constant in a team.



\begin{teo}\label{LEMPHIF}
Let $\mathcal F\in \FUN$ be a function component over some signature $\sigma$.
\begin{enumerate}
\item[(i)] For any  generalized causal team $T$ over $\sigma$,  we have that 
\[
T\models^g \Phi^\mathcal F  \iff \text{for all } (s,\mathcal G)\in T: \mathcal G \sim \mathcal F\iff T^\F=T. 
\]

\item[(ii)] For any nonempty 
causal team $T = (T^-, \G)$ over $\sigma$, we have that 

\begin{center}\(
T\models^c \Phi^\mathcal F  \iff \G \sim \F.  
\) \end{center}
\end{enumerate}
\end{teo}
\begin{proof}
(i). The second ``$\iff$" clearly holds. We now show the first ``$\iff$". If $T=\emptyset$, the equivalence holds by the empty team property. Now assume that $T\neq\emptyset$.

$\Longrightarrow$: Suppose $T\models^g \Phi^\mathcal F$ and $(s,\mathcal G)\in T$. We show that then $\mathcal G \sim \mathcal F$. \\
$\End(\mathcal F)\setminus \Con(\mathcal F) \subseteq \End(\mathcal G)\setminus \Con(\mathcal G)$: For any $V\in \End(\mathcal F)\setminus \Con(\mathcal F)$, there are distinct $\SET p,\SET p'\in \Ran(PA_V^{\mathcal F})$ such that $\mathcal F_V(\SET p) \neq \mathcal F_V(\SET p')$. Since $T\models\eta_\sigma(V)$, for any $\SET w \in Ran(\SET W)$, we have that 
\begin{center}$
\begin{array}{l}
\{(s,\mathcal G)\} \models (\SET W = \SET w \land PA_V^\mathcal F = \SET p) \cf V=\mathcal F_V(\SET p),\\
\{(s,\mathcal G)\} \models (\SET W = \SET w \land PA_V^\mathcal F =\SET p') \cf V=\mathcal F_V(\SET p').
\end{array}
$
\end{center}
where $\SET W = \Dom \setminus ({PA_V^\F \cup\{V\}})$. 
Thus,
\[ 
s^\G_{\SET W = \SET w \land PA_V^{\mathcal F} = \SET p}(V)=\mathcal F_V(\SET p)\neq \mathcal F_V(\SET p')=s^\G_{\SET W = \SET w \land PA_V^{\mathcal F}= \SET p'}(V).
\]
So, $V$ is neither in $\Con(\G)$ nor exogenous (since the value of such variables is not affected by interventions on different variables). 
Thus, $V\in \End(\G)\setminus \Con(\G)$.

$\End(\mathcal G)\setminus \Con(\mathcal G) \subseteq \End(\mathcal F)\setminus \Con(\mathcal F)$: For any $V\in \End(\mathcal G)\setminus \Con(\mathcal G)$, there are distinct $\SET p,\SET p'\in Ran(PA_V^{\mathcal G})$ such that $\mathcal G_V(\SET p) \neq \mathcal G_V(\SET p')$.  Now, if $V\notin \End(\mathcal F)\setminus \Con(\F)$, then $T\models\xi_\sigma(V)$. Let $v=s(V)$ and $\SET Z=\SET W_V \setminus PA_V^\mathcal{G}$. Since $\{(s,\mathcal G)\}\models V=v$ and $V\notin PA_V^{\mathcal G}$, for any  $\SET z \in \Ran(\SET Z)$, we have that 
\begin{center}$
\begin{array}{l}
\{(s,\mathcal G)\} \models (\SET Z = \SET z \land PA_V^\mathcal G = \SET p) \cf V=v,\\
\{(s,\mathcal G)\} \models (\SET Z = \SET z \land PA_V^\mathcal G =\SET p') \cf V=v.
\end{array}
$\end{center}
By the definition of intervention, we must have that
\[ v=s^\G_{\SET Z = \SET z \land PA_V^{\mathcal G} = \SET p}(V)=\mathcal G_V(\SET p)\neq \mathcal G_V(\SET p')=s^\G_{\SET Z = \SET z \land PA_V^{\mathcal G}= \SET p'}(V)=v,\]
which is impossible. Hence, $V\in \End(\mathcal F)\setminus \Con(\F)$.

$\mathcal F_V \sim \mathcal G_V$ for any $V\in \End(\mathcal F)\setminus \Con(\mathcal F)$: 
Since $T\models\eta_\sigma(V)$ and $V\notin PA_V^{\G}$, for all $\SET x\in \Ran(PA_V^{\mathcal F}\cap PA_V^{\mathcal G})$, $\SET y\in \Ran(PA_V^{\mathcal F}\setminus PA_V^{\mathcal G})$, $\SET z \in \Ran(PA_V^{\mathcal G}\setminus PA_V^{\mathcal F})$ and   $\SET w\in \Ran(\SET W)$ with $\SET w\upharpoonright (PA_V^{\mathcal G}\setminus PA_V^{\mathcal F}) = \SET z$ we have that
\begin{center}\(\{(s,\mathcal G)\}\models (\SET W = \SET w \land PA_V^{\mathcal F} = \SET x\SET y)\cf V = \mathcal F_V(\SET x\SET y).\)
\end{center}
 Note that $(\G_{\SET W = \SET w \land PA_V^{\mathcal F} = \SET x\SET y})_V = \G_V$, as $V\notin \SET W\cup PA_V^\F$. Thus, 
\[ 
\mathcal F_V(\SET x\SET y)= s^\G_{\SET W = \SET w \land PA_V^{\mathcal F} = \SET x\SET y}(V)=\mathcal G_V(s^\G_{\SET W = \SET w \land PA_V^{\mathcal F}= \SET x\SET y}(PA_V^\mathcal G)) =\mathcal G_V(\SET x\SET z),
\]
as required. 

$\Longleftarrow$: Suppose that $\mathcal G \sim \mathcal F$ for all $(s,\mathcal G)\in T$. Since the formula $\Phi^{\F}$ is flat, it suffices to show that $\{(s,\mathcal G)\}\models \eta_\sigma(V)$ for all $V\in \End(\mathcal F)\setminus \Con(\F)$, and $\{(s,\mathcal G)\}\models \xi_\sigma(V)$ for all $V\notin \End(\mathcal F))\setminus \Con(\F)$.

For the former, take any  $\SET w\in \Ran(\SET W)$ and $\SET p\in \Ran(PA_V^\mathcal F)$, and let  $\SET Z = \SET z$ abbreviate $\SET W = \SET w \land PA_V^{\mathcal F} = \SET p$. We show that $\{(s_{\SET Z = \SET z},\mathcal G_{\SET Z = \SET z})\}\models V=\mathcal F_V(\SET p)$. 
Since $\mathcal G \sim \mathcal F$, by Fact \ref{LEMINTSIM1} we have that $ \mathcal G_{\SET Z = \SET z}\sim \mathcal F_{\SET Z = \SET z}$. Thus, 
\begin{center}$\begin{array}{rlr}
s^\G_{\SET Z = \SET z}(V)= (\mathcal G_{\SET Z = \SET z})_V(s^\G_{\SET Z = \SET z}(PA_V^{\mathcal G_{\SET Z = \SET z}})) & = (\mathcal F_{\SET Z = \SET z})_V(s^\G_{\SET Z = \SET z}(PA_V^{\mathcal F_{\SET Z = \SET z}})) & (\text{since }\mathcal G_{\SET X= \SET x} \sim \mathcal F_{\SET X= \SET x})
\\
 & = \mathcal F_V(s^\G_{\SET Z = \SET z}(PA_V^{\mathcal F}))&(\text{since } V\notin \SET Z)\\
 & =\mathcal F_V(\SET p).&
\end{array}
$\end{center}

For the latter, take any $v\in \Ran(V)$ and $\SET w\in \Ran(\SET W_V)$. Assume that $\{(s,\mathcal G)\}\models V=v$, i.e., $s(V)=v$. Since $V\notin \End(\mathcal F)\setminus \Con(\F)$ and $\mathcal F\sim\mathcal G$, we know that $V\notin \End(\mathcal G)$ or $V\in \Con(\mathcal G)$. In both cases we have that $\{(s,\mathcal G)\}\models \SET W_V=\SET w \cf V=v$.

(ii). Let $T$ be a nonempty causal team. Consider its associated generalized causal team $T^g$. The claim then follows from Lemma \ref{LEMIDENTIFY} and item (i). 
\end{proof}


We call a generalized causal team $T$  \textbf{uniform} if  for all $(s,\mathcal F),(t,\mathcal G)\in T$, $\mathcal F\sim \mathcal G$. In particular, the empty team $\emptyset$ is uniform. 
Uniformity is easily definable in $\COv$ by using the characterizing formulas $\Phi^\F$ of the various function components.

\begin{coro}\label{LEMCHARCT}
For any generalized causal team  $T$ over some signature $\sigma$,
\[
\displaystyle T\models^g \bigvvee_{\mathcal F \in \FUN}\Phi^\mathcal F\iff  T \text{ is uniform}.  
\]
\end{coro}

\noindent For any causal team $T=(T^-,\F)$, the associated generalized causal team $T^g=$ $=\{(s,\F)\mid s\in T^-\}$ is clearly uniform. It thus follows from Lemma \ref{LEMIDENTIFY}(i) that the formula $\bigvvee_{\mathcal F \in \FUN}\Phi^\mathcal F$ is valid in causal team semantics.

\begin{coro}\label{LEM_unf_ct}
$\models^c \bigvvee_{\mathcal F \in \FUN}\Phi^\mathcal F$.
\end{coro}

\noindent Furthermore, we now show that the two entailment relations $\models^g$ and $\models^c$ are equivalent modulo the formula that defines uniformity.

\begin{lm}\label{LEMCTMODELS}
For any set $\Gamma\cup\{\psi\}$ of \CO- or \COV- or $\COD$-formulas, 
\[\displaystyle
\Gamma\models^{c} \psi \iff  \displaystyle\Gamma,\bigvvee_{\mathcal F\in \FUN} \Phi^{\mathcal F}\models^{g} \psi. \]
\end{lm}
\begin{proof}
$\Longleftarrow$: Suppose $T\models^{c} \Gamma$ for some causal team $T$. Consider the generalized causal team $T^{g}$ generated by $T$.  By Lemma \ref{LEMIDENTIFY}(i),  $T^{g}\models^{g}\Gamma$. Since $T^{g}$ is uniform,  Corollary \ref{LEMCHARCT} gives that $T^{g}\models^{g} \bigvvee_{\mathcal F\in \FUN}\Phi^{\mathcal F}$. Then, by assumption, we obtain that $T^{g}\models^{g}\psi$, which, by Lemma \ref{LEMIDENTIFY}(i) again, implies that $T\models^{c}\psi$.

$\Longrightarrow$:   Suppose $T\models^{g}\Gamma$ and $T\models^{g}\bigvvee_{\mathcal F\in \FUN}\Phi^{\mathcal F}$ for some generalized causal team $T$. If $T=\emptyset$, then $T\models^g\psi$ by the empty team property. Now assume that $T\neq\emptyset$. By Corollary \ref{LEMCHARCT} we know that $T$ is uniform. Pick $(t,\mathcal F)\in T$.  
%
Consider the generalized causal team $S=\{(s,\F)\mid s\in T^-\}$. Observe that $T\approx S$. Thus, by Theorem \ref{PROPEQUIV}, we have that $S\models^{g}\Gamma$, which further implies, by Lemma \ref{LEMIDENTIFY}(ii), that $S^c\models^{c}\Gamma$. Hence, by the assumption we conclude that $S^c\models^c\psi$. 
Since $(S^c)^g=S\approx T$, by applying Lemma \ref{LEMIDENTIFY}(ii) and Theorem \ref{PROPEQUIV} again, we obtain $T\models^g\psi$.
%
%
%
\end{proof}

The $\Phi^\F$ formulas also help us in understanding the behaviour of the global disjunction $\vvee$ in our languages.  In propositional inquisitive logic (\cite{CiaRoe2011}) and  propositional dependence logic (\cite{YanVaa2016}), the global disjunction
was shown to have the {\em disjunction property}, i.e., $\models \varphi\vvee\psi$ implies $\models\varphi$ or $\models\psi$.
It follows immediately from Theorem \ref{LEMPHIF} that the disjunction property of $\vvee$ fails in the context of causal teams, because  $\models^c\bigvvee_{\mathcal F \in \FUN}\Phi^\mathcal F$, whereas $\not\models^c\Phi^\mathcal F$ for any $\mathcal F \in \FUN$. Nevertheless, the global disjunction  does admit the disjunction property over generalized causal teams.

\begin{teo}[Disjunction property]
\label{splitting_prop}
Let $\Delta$ be a set of $\CO$-formulas, and $\varphi,\psi$ be arbitrary formulas over $\sigma$. If $\Delta\models^{g}\varphi\vvee\psi$, then $\Delta\models^{g}\varphi$ or $\Delta\models^{g}\psi$. 
In particular, if $\models^{g}\varphi\vvee\psi$, then $\models^{g}\varphi$ or $\models^{g}\psi$.  
\end{teo}
\begin{proof}
Suppose $\Delta\not\models^{g}\varphi$ and $\Delta\not\models^{g}\psi$. Then there are two generalized causal teams $T_1,T_2$ such that $T_1\models \Delta$, $T_2\models \Delta$, $T_1\not\models\varphi$ and $T_2\not\models\psi$. Let $T:= T_1\cup T_2$. 
By flatness of $\Delta$, we have that $T\models\Delta$. On the other hand, by downwards closure, we have that $T\not\models\varphi$ and $T\not\models\psi$, and thus $T\not\models\varphi\vvee \psi$.
\end{proof}


Regarding causal team semantics (an \emph{not} the generalized semantics) there is another peculiarity in the behaviour of the global disjunction $\vvee$:  when applied to causally incompatible formulas, the global disjunction $\vvee$ ``collapses'' to tensor disjunction $\vee$. We say that two formulas $\varphi$ and $\psi$ in $\COD$ or $\COV$ are \textbf{(causally) incompatible} if, for all nonempty causal teams $S =(S^-,\F)$ and $T = (T^-,\G)$, 
$$S\models\varphi \text{ and }  T\models\psi \implies \F\not\sim\G.$$
For instance, it follows immediately from Theorem \ref{LEMPHIF} that the formulas $\Phi^\F$ and $\Phi^\G$ are incompatible when  $\F\not\sim\G$.  The key formula in Corollary \ref{LEMCHARCT}, $\bigvvee_{\F\in \FUN}\Phi^\F$, is thus a global disjunction of pairwise incompatible formulas $\Phi^\F$. By the result below, we will have $\bigvvee_{\F\in \FUN}\Phi^\F\equiv^c\bigvee_{\F\in \FUN}\Phi^\F$, although  the same result does not apply to generalized causal teams (as $\bigvvee_{\F\in \FUN}\Phi^\F\not\equiv^g\bigvee_{\F\in \FUN}\Phi^\F$).

%

\begin{lm}\label{LEMORCOLLAPSE}
For any set $\{\varphi_i\mid i\in I\}$  of pairwise incompatible $\COD$- or $\COV$-formulas, we have that 
\(\bigvvee_{i\in I}\varphi_i\equiv^c\bigvee_{i\in I}\varphi_i.\)
\end{lm}

\begin{proof}
The  left-to-right entailment follows easily from the empty team property. 
Let us now show the right-to-left entailment.
Suppose $T\models^c\bigvee_{i\in I}\varphi_i $. If $T=\emptyset$, then $T\models^c \bigvvee_{i\in I}\varphi_i$ by the empty team property. Now assume that $T\neq\emptyset$. Then there are causal subteams $S_i =(S_i^-,\F)$ of $T$  such that $\bigcup_{i\in I} S_i^- = T^-$ and $S_i\models \varphi_i$ for each $i\in I$. Since $T^-\neq\emptyset$, at least one $S_i$ is nonempty. If both  $S_i=(S_i^-,\F)$ and $S_j=(S_j^-,\F)$ were  nonempty ($i\neq j$), then since $\varphi_i$ and $\varphi_j$ are incompatible, we would have that $\F\not\sim\F$; a contradiction. Hence, we conclude that exactly one $S_i$ is nonempty, and thus $S_i=T$. Therefore $T\models\varphi_i$, from which $T\models\bigvvee_{i\in I}\varphi_i$ follows. 
%
%
\end{proof}


The above ``collapse'' does not only occur in causal team semantics; for a simple example: 
$\alpha\vee\alpha\equiv^{c/g}\alpha\equiv^{c/g}\alpha\vvee\alpha$ for any \Co-formula $\alpha$ (e.g., $\alpha=\Phi^{\F}$).
However, it easy to show that in the generalized (as well as in the non-causal) team semantics, the entailment $\varphi\lor\psi \models^g \varphi\vvee\psi$ 
implies that $\varphi\models\psi$ or $\psi\models\varphi$ (provided $\varphi$ and $\psi$ are downward closed).\footnote{For contraposition, suppose that $\varphi\not\models\psi$ and $\psi\not\models\varphi$. Then there are generalized causal teams $T,S$ such that $T\models \varphi, T\not\models \psi, S\models\psi$ and $S\not\models\varphi$. Thus, we conclude $T\cup S\models\varphi\vee\psi$ by definition, and $T\cup S\not\models\varphi$ and $T\cup S\not\models\psi$ by the downward closure of $\varphi$ and $\psi$. Hence $\varphi\vee\psi\not\models\varphi\vvee\psi$.
} 
On the contrary, the equivalence $\Phi^{\F}\vee\Phi^{\G}\equiv^{c}\Phi^{\F}\vvee\Phi^{\G}$,  for $\F\not\sim\G$, which holds in causal team semantics, does not respect this necessary condition, as neither $\Phi^{\F}$ implies $\Phi^{\G}$ nor vice versa.
\section{Characterizing $\Co$}\label{SECCODED}

In this section, we 
characterize the expressive power of $\Co$ over causal teams. We also present a system of natural deduction for $\Co$ that is sound and complete over both (recursive) causal teams and generalized causal teams. The system is equivalent to the Hilbert-style system of the same logic over (recursive) causal teams introduced in \cite{BarSan2019}.

\subsection{Expressivity}\label{SECCOEXPR}

We show  that $\Co$-formulas capture the flat class of (generalized) causal teams, up to $\approx$-equivalence. 
Our result is analogous to known characterizations of flat languages in propositional team semantics (\cite{YanVaa2017}). But in the case of causal teams we have a twist, given by the fact that only the unions of similar causal teams can be  reasonably defined. We define such unions as follows.

\begin{df}\label{union_ct_df}
Let $S=(S^-,\F)$ and $T=(T^-,\G)$ be two causal teams over the same signature $\sigma$ with $S\sim T$. The union of $S$ and $T$ is defined as the causal team $S\cup T=(S^-\cup T^-,\mathcal H)$ over $\sigma$, where 
\begin{itemize}
\item[-] $\End(\mathcal H)=\End(\F)\setminus \Con(\F)$ 

\item[-]  and for each $V\in \End(\mathcal H)$, $PA^{\mathcal H}_V=PA^{\F}_V\cap PA^{\G}_V$, and
\(\mathcal H_V(\SET p)=\F_V(\SET p\SET x)\)
for any $\SET p\in PA_V^{\F}\cap PA_V^{\G}$ and $\SET x \in PA_V^{\F}\setminus PA_V^{\G}$. 
\end{itemize}
\end{df}
Clearly, $\mathcal H\sim\F\sim\G$ and thus $S\cup T\sim S\sim T$. Let us also remark that even though the definition may seem to look asymmetric as it appears to rely more on $\F$ than on $\G$ in its details, the notion of union is indeed  well-defined, and it guarantees that $S\cup T = T\cup S$. This is because the definition requires that $\F\sim \G$, from which it follows that $\End(\mathcal H) =\End(\F)\setminus \Con(\F)= \End(\G)\setminus \Con(\G)$, and also that $\mathcal H_V(\SET p)=\F_V(\SET p\SET x)=\G_V(\SET p\SET y)$ for any $\SET p\in PA_V^{\F}\cap PA_V^{\G}$, $\SET y \in PA_V^{\G}\setminus PA_V^{\F}$ and $\SET x \in PA_V^{\F}\setminus PA_V^{\G}$.

A formula $\varphi$ over a signature 
$\sigma$ determines a class $\K_\varphi$ of (generalized)
causal teams over $\sigma$ defined as
\begin{center}
\(
\K_\varphi=\{T\mid T \text{ has signature $\sigma$ and } T\models\varphi\}.
\)
\end{center}
We say that a formula $\varphi$ \textbf{defines} a class $\K$ of  (generalized) 
causal teams over $\sigma$ if $\K=\K_\varphi$.

\begin{df}
We say that  a class  $\K$ of (generalized) 
causal teams over  $\sigma$ is
\begin{itemize}
\item[-] \textbf{causally downward closed} if $T\in \K$ and $S\subseteq T$ imply $S\in\K$;

\item[-] \textbf{closed under causal unions} if, whenever  $T_1,T_2\in\K$ and $T_1\cup T_2$ is defined (which is always the case for generalized causal teams), 
 $T_1\cup T_2\in \K$;

\item[-] \textbf{flat} if 
	\begin{itemize}
	\item (for causal teams) $(T^-,\F)\in \K$ iff $(\{s\},\F)\in \K$ 		for all $s\in T^-$; 
	\item (for generalized causal teams) $T\in \K$ iff $\{(s,\F)\}\in 		\K$ for all $(s,\F)\in T$.
	\end{itemize}
\item[-] \textbf{closed under  equivalence}  if $T\in \K$ and $T\approx S$ imply $S\in \K$. 
\end{itemize}
\end{df}
\begin{fct}
A 
class $\K$ of (generalized) causal teams over  $\sigma$  is flat iff $\K$ contains the empty team,  
is causally downward closed and closed under causal unions. 
\end{fct}
The class $\K_\varphi$ is always nonempty as the empty team is 
always in $\K_\varphi$ (by Theorem \ref{TEOGENDW}). By Theorems \ref{TEOGENDW} and \ref{PROPEQUIV}, if $\alpha$ is a $\Co$-formula, then $\K_\alpha$ is flat and closed under equivalence.
The main result of this section is the following characterization theorem which gives also the converse direction.

\begin{teo}\label{TEOCHARCOCT}
Let $\K$ be a nonempty class of (generalized) 
causal teams over some signature $\sigma$. Then $\K$ is definable by a $\CO$-formula if and only if $\K$ is flat 
and closed under equivalence.
\end{teo}

\noindent Before we present the proof of the theorem, let us first introduce a relation $\preccurlyeq$ between causal teams that is of central use in this proof.
For any (generalized) causal teams $S$ and $T$, we write $S\preccurlyeq T$ if $S\approx R\subseteq T$ for some (generalized) causal team $R$. In particular, we stipulate $\emptyset\preccurlyeq T$ for all (generalized) causal teams $T$. 

\begin{fct}\label{fct_preccurlyeq_closure}
A class $\K$ of (generalized) causal teams over  $\sigma$  is causally downward closed and closed under equivalence iff $\K$ is closed under the relation $\succcurlyeq$, i.e., $T\in \K$ and $T\succcurlyeq S$ imply $S\in \K$.
\end{fct}

We leave it for the reader to verify that $\preccurlyeq$ is a partial order (modulo $\approx$); in particular, $S\preccurlyeq T$ and $T\preccurlyeq S$ imply $S\approx T$. 
The next lemma lists some other useful properties of the relation $\preccurlyeq$, 
where recall that, for every generalized causal team $T$, we define $T^\mathcal F = \{(s,\mathcal G)\in T \mid \mathcal G\sim \mathcal F \}$. 


%


\begin{lm}\label{preccurlyeq_prop}
The following holds in the context of generalized causal teams:
\begin{enumerate}
\item[(i)]\label{preccurlyeq_prop_simplified} $S\preccurlyeq T \iff \{(s,\F)\}\preccurlyeq T$ for all $(s,\F)\in S$.
\item[(ii)]\label{FACTEQPRESERVESUNIFORMITY} If $T=T^\F$ is uniform and $S\preccurlyeq T$, then $S$ is also uniform and $S=S^\F$.
\item[(iii)]\label{preccurlyeq_prop_clo_F} If $S\preccurlyeq T$, then $S^\F\preccurlyeq T^\F$ for all $\F\in \FUN$.
\item[(iv)]\label{preccurlyeq_prop_c2g} If $S,T$ are causal teams, then $S\preccurlyeq T\iff S^g\preccurlyeq T^g$.
\item[(v)]\label{preccurlyeq_prop_union} If $S_i\preccurlyeq T_i$ ($i\in I$), 
then $\bigcup_{i\in I}S_i\preccurlyeq \bigcup_{i\in I}T_i$
\item[(vi)]\label{preccurlyeq_prop_fromunion} If $S\preccurlyeq \bigcup_{i\in I}T_i$, then for every $i\in I$, there exists $S_i\preccurlyeq T_i$ such that $S=\bigcup_{i\in I}S_i $.
\end{enumerate}
\end{lm}
\begin{proof}
Item (i) 
is easy to prove,
items (ii) 
and (iii) 
are obvious, and item (iv) 
follows easily from the fact that both $S^g$ and $T^g$ are uniform. We only give detailed proof for items (v) 
and (vi).

Let us begin with (v). 
Suppose that $S_i\approx R_i\subseteq T_i$ for some $R_i$ ($i\in I$). Let $R=\bigcup_{i\in I}R_i$. Clearly,  $\bigcup_{i\in I}R_i\subseteq \bigcup_{i\in I}T_i$. It remains to show that $S\approx R$ for $S=\bigcup_{i\in I}S_i$. For every $\F\in\FUN$ and every $i\in I$, since $S_i\approx R_i$, we have that $(S_i^\F)^-=(R_i^\F)^-$. Thus, 
\[
(S^\F)^-=\bigcup_{i\in I}(S_i^\F)^-=\bigcup_{i\in I}(R_i^\F)^-=(R^\F)^-.
\]


Now we prove (vi). 
By the assumption, there is an $R$ such that $S\approx R\subseteq T=\bigcup_{i\in I}T_i$. For every $i\in I$, define $R_i:= R\cap T_i$ and 
\[S_i:=\{(s,\F)\in S \mid \{(s,\F)\}\preccurlyeq  R_i\}.\] 
Clearly, $R_i\subseteq T_i$. To see that $S_i\preccurlyeq T_i$, it remains to verify that $S_i\approx R_i$. For any $\F\in \FUN$, we have that 
\begin{align*}
s\in (S_i^\F)^-&\iff (s,\G)\in S_i\text{ for some }\G\sim\F\iff \{(s,\G)\}\preccurlyeq  R_i\text{ for some }\G\sim\F\\
&\iff \{(s,\G)\}\preccurlyeq R_i^\F\text{ for some }\G\sim\F\iff s\in (R_i^\F)^-.
\end{align*}
Finally, we show that $S=\bigcup_{i\in I}S_i$. By definition, $S_i\subseteq S$ for every $i\in I$. Conversely, for every $(s,\F)\in S$, since $S\approx R$,  
we have by item (i) that $$\{(s,\F)\}\preccurlyeq R=R\cap T=\bigcup_{i\in I}(R\cap T_i)=\bigcup_{i\in I}R_i.$$ Thus,  $\{(s,\F)\}\preccurlyeq R_i$ for some $i\in I$. Hence, by definition, $(s,\F)\in S_i$. 
%
\end{proof}

We will next introduce a formula that characterizes the relation $\preccurlyeq$ in a certain sense. An important part of this characterizing formula is a \CO-formula $\Theta^{T^-}$ that defines the property ``having as team component a subset of $T^-$''. This formula is inspired by a similar one introduced in \cite{YanVaa2016} in the context of propositional team logics. For each (generalized) causal team $T$ over $\sigma=(\Dom,\Ran)$, define
\begin{center}\(\displaystyle
\Theta^{T^-} := \bigvee_{s\in T^-} \bigwedge_{V\in \Dom} V = s(V). 
\)\end{center}
Note that if $T^- = \emptyset$, then $\Theta^{T^-} =  \bigvee\emptyset \ = \ \bot$.

\begin{lm} \label{LEMCHARSUB_ct}
For any (generalized) causal teams $S$ and $T$ over  $\sigma$, we have that \(S\models\Theta^{T^-}\iff S^-\subseteq T^-\).
\end{lm}
\begin{proof}
We only give detailed proof for the statement for causal teams; the generalized causal team case is analogous.

``$\Longrightarrow$'': Suppose $S\models\Theta^{T^-}$ and $S=(S^-,\F)$. 
For any $s\in S^-$, by downward closure, we have that $(\{s\},\F)\models \Theta^{T^-}$, which means that for some $t\in T^-$, \((\{s\},\F)\models V = t(V)\) for all $V\in\Dom$.
This implies that $s=t\in T^-$.

``$\Longleftarrow$'': Suppose $S^-\subseteq T^-$. Observe that 
$S\models\Theta^{S^-}$ and $\Theta^{T^-}=\Theta^{S^-}\vee \Theta^{T^-\setminus S^-}$.
Thus, we conclude $S\models \Theta^{T^-}$ by the empty team property.
\end{proof}

Next, we show that the formula $\Theta^{T^-}\wedge \Phi^{\F}$ defines the property of being in the  $\preccurlyeq$ relation with a uniform generalized causal team $T=T^\F$ or a  causal team $(T^-,\F)$.

\begin{lm}\label{LEMCHARSUB}
\begin{enumerate}
    \item[(i)] Let $S$ and $T$ be 
    generalized causal teams over $\sigma$. If $T=T^\F$ is uniform, then
\(
S\models^g\Theta^{T^-}\wedge \Phi^{\F} \iff S\preccurlyeq T.
\)
\item[(ii)] For any 
causal teams $S=(S^-,\G)$ and $T=(T^-,\F)$ over $\sigma$, we have that
\(
S\models^c\Theta^{T^-}\wedge \Phi^{\F} \iff S\preccurlyeq T.
\)
\end{enumerate}
\end{lm}

\begin{proof}
(i)  If $S=\emptyset$, then $\emptyset \models^g\Theta^{T^-}\wedge \Phi^{\F} $ by the empty team property, and $\emptyset\preccurlyeq T$ by stipulation. If $T=\emptyset$, then $\Theta^{T^-}=\bigvee\emptyset=\bot$, and thus $\Theta^{T^-}\wedge \Phi^{\F}\equiv\bot$. Obviously, $S\models^g\bot$ iff $S=\emptyset$ iff $S\preccurlyeq \emptyset=T$. Now assume that $S,T\neq\emptyset$.

``$\Longrightarrow$'': 
Consider 
\(R=\{(s,\G)\in T\mid s\in S^-\}\subseteq T.\) It suffices to show that $S\approx R$, i.e., $(S^{\G})^-=(R^{\G})^-$ for all $\G\in \FUN$. 
If $\G\not\sim\F$, 
since $T=T^\F$ we have that $(R^{\mathcal{G}})^-\subseteq (T^{\mathcal{G}})^-=\emptyset$. On the other hand, since $S\models\Phi^{\F}$, by Theorem \ref{LEMPHIF} we must have that $(S^{\mathcal{G}})^-=\emptyset$. Thus, $(R^{\mathcal{G}})^-=\emptyset=(S^{\mathcal{G}})^-$.
Now, if $\G\sim\F$, then $R^{\mathcal{G}}=R^\F=R$, since $R\subseteq T=T^\F$. As $S\models\Theta^{T^-}$, by Lemma \ref{LEMCHARSUB_ct} we have that $S^-\subseteq T^-$, which, by definition of $R$, implies that $R^-=S^-$. Moreover, since $S\models \Phi^{\F}$, by Theorem \ref{LEMPHIF} we have that $S=S^{\F}=S^{\mathcal{G}}$. Thus, we conclude that $(S^{\mathcal{G}})^-=S^-=R^-=(R^{\mathcal{G}})^-$.

``$\Longleftarrow$'': Suppose $S\approx R\subseteq T=T^\F$. Then $R^\F=R\approx S$. It follows by Lemma \ref{preccurlyeq_prop}(ii) 
that $S^{\F}=S$ and $S^-=(S^\F)^-=(R^\F)^-=R^-\subseteq T^-$. 
Thus, by Theorem \ref{LEMPHIF} and Lemma \ref{LEMCHARSUB_ct} we obtain that $S\models\Phi^\F\wedge\Theta^{T^-}$.

(ii)  Observe that the generalized causal team $T^g=\{(s,\F)\mid s\in T^-\}$ associated with $T$ is uniform, $T^g=(T^g)^\F$ and $T^-=(T^g)^-$. Thus, by Lemma \ref{LEMIDENTIFY}(i), item (i) and Lemma \ref{preccurlyeq_prop}(iv), 
\[
S\models^c\Theta^{T^-}\wedge \Phi^{\F} \iff S^g\models^g\Theta^{(T^g)^-}\wedge \Phi^{\F} \iff S^g\preccurlyeq T^g\iff S\preccurlyeq T.
\]
%
%
\end{proof}


\begin{coro}\label{COROCHARSUB}
For any 
generalized causal teams $S$ and $T$ over $\sigma$, we have that
\[
S\models \bigvee_{\F\in\FUN}(\Theta^{(T^\F)^-}\wedge \Phi^{\F}) \iff S\preccurlyeq T.
\] 

\end{coro}

\begin{proof}
Put $\displaystyle\psi= \bigvee_{\F\in\FUN}(\Theta^{(T^\F)^-}\wedge \Phi^{\F})$. 
If $T=\emptyset$, then $T^\F=\emptyset$ for all $\F\in\FUN$, so that $\Theta^{(T^\F)^-} \equiv^g \bot$ and thus $\psi\equiv^g\bot$. Then, $S\models\psi$ iff $S=\emptyset$ iff $S \preccurlyeq \emptyset=T$.

Now, assume that $T\neq\emptyset$. For each $\F\in \FUN$, the team $T^\F$ is uniform. Thus, 
\begin{align*}
 &S\models \bigvee_{\F\in\FUN}(\Theta^{(T^\F)^-}\wedge \Phi^{\F})\\
 \iff& \forall \F\in \FUN\exists S_\F: S=\bigcup_{\F\in \FUN}S_\F\text{ and each }S_\F\models \Theta^{(T^\F)^-}\wedge \Phi^{\F}\\
 \iff& \forall \F\in \FUN\exists S_\F: S=\bigcup_{\F\in \FUN}S_\F\text{ and each }S_\F\preccurlyeq T^\F\tag{by Lemma \ref{LEMCHARSUB}}\\
 \iff& S\preccurlyeq T.\tag{by Lemma \ref{preccurlyeq_prop}(v,vi), 
 since $\displaystyle T=\bigcup_{\F\in \FUN}T^\F$}
 \end{align*}
\end{proof}

We are now ready to prove the main theorem of the section.

\vspace{5pt}

\begin{lateproof}{Theorem \ref{TEOCHARCOCT}}
It suffices to prove the direction ``$\Longleftarrow$". We first consider the generalized causal team case. Let $\K$ be a nonempty class of generalized causal teams that is flat and closed under equivalence. Put $T=\bigcup \K$. Since $\K$ is closed under unions, $T\in \K$. Consider the $\CO$-formula 
\begin{equation*}
    \varphi= \bigvee_{\F\in\FUN}(\Theta^{(T^\F)^- }\land \Phi^{\F }).
\end{equation*}
We show that $\K_\varphi=\K$. 
For any generalized causal team $S$, we first deduce, by applying Corollary \ref{COROCHARSUB}  that
\[S\in \K_\varphi\iff S\models^g\varphi\iff S\preccurlyeq \bigcup \K.\]
Now assume $S\in\K_\varphi$; since then $S\preccurlyeq \bigcup \K$, and since $\bigcup \K\in\K$ and $\K$ is closed under $\succcurlyeq$, we have that 
$S\in \K$, which gives $\K_\varphi\subseteq\K$. 
Conversely, if $S\in \K$, then obviously $S\approx S\subseteq  \bigcup \K$, namely, $S\preccurlyeq \bigcup \K$, from which we conclude that $S\in \K_\varphi$. 


Next, we consider the causal team case. Let $\K$ be a nonempty class of causal teams that is flat and closed under equivalence. For each $\F\in \FUN$, let \[(T_\F)^-=\bigcup\{T^-\mid (T^-,\G)\in\K\text{ for some }\G\sim\F\}.\]
Since $\K$ is closed under unions and equivalence, $((T_\F)^-,\F)\in \K$.
Consider the {\em generalized} causal team $T=\bigcup\K^g$, where \(\K^g=\{T^g\mid T\in \K\}\). Observe that $(T^\F)^-=(T_\F)^-$.
Consider the above-defined $\CO$-formula $\varphi$ with respect to the generalized causal team $T$. Thus, $\K_\varphi=\K$ over causal teams, since for any causal team $S=(S^-,\G)$,
%
\begin{align*}
    S\models^c\varphi&\iff S^g\models^g\varphi\tag{Lemma \ref{LEMIDENTIFY}(i)}\\
    &\iff S^g\preccurlyeq T\tag{Corollary \ref{COROCHARSUB}}\\
    &\iff S^g=(S^g)^{\G} \preccurlyeq T^{\G} \tag{Lemma \ref{preccurlyeq_prop}(iii)} 
    \\
    &\iff S\preccurlyeq ((T_\G)^-,\G)
    \tag{by Lemma \ref{preccurlyeq_prop}(iv), since $(T^\G)^-=(T_\G)^-$}\\
     &\iff S\in \K.\tag{since $((T_\G)^-,\G)\in\K$ and $\K$ is closed under $\succcurlyeq$}
\end{align*}
\end{lateproof}

\noindent In the proof above, the case for causal teams is obtained by a reduction to the generalized case. For a direct proof, see \cite{BarYan2020}.



We end this section with some remarks on the sublanguage of $\Co$ without the operator $\cf$. 
We show that this language is \emph{strictly} less expressive than $\Co$ (and similarly for $\COv$ and $\COd$).

\begin{lm}\label{LEMWOCF}
Let $\varphi$ be a $\COV$ or  $\COD$ formula without occurrences of $\cf$. For any (generalized) causal teams $T,S$ over $\sigma$ with $T^-=S^-$, $T\models \varphi \iff S\models \varphi$.
\end{lm}

\begin{proof}
By a straightforward induction on $\varphi$.
\end{proof}

\begin{prop}
Let $\sigma=(\Dom,\Ran)$ be a signature such that $|\Ran(V)|\geq 2$ for at least one variable $V\in\Dom$. Then there is a class of (generalized) causal teams that is definable by a $\CO$ formula, but  not definable by any $\CO$ formula without occurrences of $\cf$.
\end{prop}

\begin{proof}
Let $(s,\F)\in\SEM$. Consider the  class of generalized causal teams $\K=\wp(\{(s,\G)\mid \G\sim \F\})$. 
This 
class is flat 
and closed under equivalence, so by Theorem \ref{TEOCHARCOCT} it is definable by a formula of $\CO$. On the other hand, observe that, given the cardinality assumption on $\sigma$, there is an $\mathcal H\in \FUN$ such that $\mathcal H\not\sim \F$. Now, $\{(s,\mathcal H)\}\notin\K$ and also $\{(s,\F)\}\in\K$. For any formula $\theta$ of $\CO$ without occurrences of $\cf$, by Lemma \ref{LEMWOCF}, we have that $\{(s,\mathcal H)\}\models\theta$ iff $\{(s,\F)\}\models\theta$, that is, $\{(s,\mathcal H)\}\in \K_\theta$ iff $\{(s,\F)\}\in \K_\theta$. Hence, $\K_\theta \neq \K$. 


For causal team semantics, we reason similarly by considering instead  the class of causal teams $\K=
\{\emptyset\}\cup \{(\{s\},\G)\mid \G \sim \F\}$.
\end{proof}

This proposition implies 
that causal discourse cannot be entirely reduced to the (propositional) pure team setting, in analogy with the motto that causation cannot be explained in terms of correlation.\footnote{On the other hand, the quantifiers of \emph{first-order} team semantics act as dynamic operators, and thus they allow, to some extent, to simulate the effect of interventions. This is discussed in \cite{BarGal2020}.}
A further simple consequence of the above proposition is the undefinability of $\cf$ in terms of the other connectives (see \cite{Yan2017,CiaBar2019} for details on the notion of definability of connectives in the context of team semantics).

\subsection{Deduction system}\label{SUBSAXCO}

The logic $\Co$ over (recursive) causal teams was axiomatized in \cite{BarSan2019} by means of a sound and complete Hilbert-style deduction system. 
In this section, we present an equivalent system of natural deduction and show it to be sound and complete also over (recursive) \emph{generalized} causal teams. 

Recall that we have confined ourselves to {\em recursive} (generalized) causal teams only. The deduction system we introduce in this section will therefore contain some rules that are not valid over nonrecursive causal teams (such as the rule $\cf\textsf{I}$). 
Furthermore, it will have a rule (labelled as \textsf{Recur}) that characterizes recursiveness. In order to present this rule \textsf{Recur} in a compact form, following \cite{Hal2000},
we write $X\leadsto Y$ (``$X$ causally affects $Y$'') as an abbreviation for the $\CO$ formula:
\begin{align*}
\bigvee \Big\{\SET Z &= \SET z \cf \big((X=x \cf Y=y) \land (X=x' \cf Y=y')\big)\\
&\mid \SET Z\subseteq \Dom\setminus\{X,Y\},\SET z\in \Ran(\SET Z), x,x'\in \Ran(X), y,y'\in \Ran(Y),x\neq x',y\neq y'\Big\}.
\end{align*}
This formula states that there is some (possibly null) intervention $do(\SET Z = \SET z)$ after which intervening further on $X$ may alter the value of $Y$. It is in a sense the weakest possible statement of causal dependence between $X$ and $Y$.

\begin{df}\label{co-system-df}
The system of natural deduction for $\CO$ consists of the following rules: 

\begin{itemize}
\item[-] (Parameterized) rules for value range assumptions:
\begin{center}
{\normalfont
\renewcommand{\arraystretch}{1.8}
\def\ScoreOverhang{0.5pt}
\def\defaultHypSeparation{\hskip .1in}
\begin{tabular}{|C{0.45\linewidth}C{0.45\linewidth}|}
\hline
 \AxiomC{} \noLine\UnaryInfC{} \RightLabel{$\textsf{ValDef}$}\UnaryInfC{$\bigvee_{x\in \Ran(X)} X=x$} \noLine\UnaryInfC{}\DisplayProof & \AxiomC{}\noLine\UnaryInfC{$X=x$} \RightLabel{$\textsf{ValUnq}$ \hspace{5pt} (for $x \neq x'$) }\UnaryInfC{$X\neq x'$} \noLine\UnaryInfC{}\DisplayProof\\\hline
 \end{tabular}
 }
\end{center}

\item[-] Rules for $\wedge,\vee,\neg$:

\begin{center}
{\normalfont
\renewcommand{\arraystretch}{1.8}
\begin{tabular}{|C{0.45\linewidth}C{0.45\linewidth}|}
\hline


 \AxiomC{$\varphi$}\AxiomC{$\psi$} \RightLabel{$\wedge\textsf{I}$}\BinaryInfC{$\varphi\wedge\psi$} \DisplayProof&
 
 \AxiomC{$\varphi\wedge\psi$} \RightLabel{$\wedge\textsf{E}$}\UnaryInfC{$\varphi$} \DisplayProof~
   \AxiomC{$\varphi\wedge\psi$} \RightLabel{$\wedge\textsf{E}$}\UnaryInfC{$\psi$} \DisplayProof\\
 
 
 \multirow{2}{*}{\AxiomC{$\varphi$} \RightLabel{$\vee\textsf{I}$}\UnaryInfC{$\varphi\vee\psi$}\DisplayProof~
 \AxiomC{$\varphi$} \RightLabel{$\vee\textsf{I}$}\UnaryInfC{$\psi\vee\varphi$}\DisplayProof}
 
 &
 
 \AxiomC{$\varphi\vee\psi$}\AxiomC{$[\varphi]$}\noLine\UnaryInfC{$\vdots$}\noLine\UnaryInfC{$\alpha$} \AxiomC{}\noLine\UnaryInfC{$[\psi]$}\noLine\UnaryInfC{$\vdots$}\noLine\UnaryInfC{$\alpha$}\RightLabel{$\vee\textsf{E}$} \TrinaryInfC{$\alpha$}\noLine\UnaryInfC{}\noLine\UnaryInfC{}\DisplayProof\\
 
  
 \multicolumn{2}{|c|}{\AxiomC{$[\alpha]$}\noLine\UnaryInfC{$\vdots$}\noLine\UnaryInfC{$\bot$} \RightLabel{$\neg\textsf{I}$}\UnaryInfC{$\neg\alpha$} \noLine\UnaryInfC{}\DisplayProof
\quad\quad\AxiomC{$\alpha$}\AxiomC{$\neg\alpha$}\RightLabel{$\neg\textsf{E}$}\BinaryInfC{$\varphi$} \DisplayProof
\quad\quad \AxiomC{$[\neg\alpha]$}\noLine\UnaryInfC{$\vdots$}\noLine\UnaryInfC{$\bot$}\RightLabel{\textsf{RAA}}\UnaryInfC{$\alpha$} \noLine\UnaryInfC{}\DisplayProof}\\


\hline
\end{tabular}
}
\end{center}

\item[-] Rules for $\cf$: 

{\normalfont
\renewcommand{\arraystretch}{1.7}
\def\ScoreOverhang{0.5pt}
\def\defaultHypSeparation{\hskip .03in}
\begin{center}
\begin{tabular}{|C{0.95\linewidth}|}
\hline

\AxiomC{}\noLine\UnaryInfC{}\noLine\UnaryInfC{}\noLine\UnaryInfC{} \RightLabel{{\footnotesize$\cf\!\textsf{Eff}$}}\UnaryInfC{$(\SET{X}=\SET{x} \land Y=y) \cf Y=y$} \DisplayProof
\quad\quad\quad\quad\AxiomC{}\noLine\UnaryInfC{}\noLine\UnaryInfC{$ \SET X=\SET x$\ \ \ }  \AxiomC{$\theta$}\RightLabel{\footnotesize$\cf\textsf{I}$ (1)}\BinaryInfC{$\SET X=\SET x\cf \theta$} \DisplayProof
 \\
  

 \AxiomC{}\noLine\UnaryInfC{}\noLine\UnaryInfC{}\noLine\UnaryInfC{}\RightLabel{\footnotesize$\textsf{ex falso}_{\cf}$\,(for $x\neq x'$)}\UnaryInfC{$(\SET Y=\SET y\land X=x\land X=x')\cf \varphi$}\DisplayProof\!\!\!\!\AxiomC{}\noLine\UnaryInfC{}\noLine\UnaryInfC{$\SET{X}=\SET{x}\cf \bot$}\RightLabel{\footnotesize$\cf\!\bot\textsf{E}$ (2)}\UnaryInfC{$\varphi$}\DisplayProof
%
%
\\
%
\AxiomC{$\SET{X}=\SET{x} \cf \varphi$}\AxiomC{}\noLine\UnaryInfC{}\noLine\UnaryInfC{$[\SET{X}=\SET{x}]$}\noLine\UnaryInfC{$D_1$}\noLine\UnaryInfC{$\SET{Y}=\SET{y}$}\AxiomC{}\noLine\UnaryInfC{}\noLine\UnaryInfC{$[\SET{Y}=\SET{y}]$}\noLine\UnaryInfC{$D_2$}\noLine\UnaryInfC{$\SET{X}=\SET{x}$} \RightLabel{\footnotesize$\cf\!\!\textsf{Rpl}_A$\,(3)}\TrinaryInfC{$\SET{Y}=\SET{y} \cf \varphi$}
\DisplayProof
\quad \AxiomC{$\SET{X}=\SET{x} \cf \varphi$}\AxiomC{}\noLine\UnaryInfC{}\noLine\UnaryInfC{$[\varphi]$}\noLine\UnaryInfC{$D$}\noLine\UnaryInfC{$\psi$} \RightLabel{\footnotesize$\cf\!\!\textsf{Rpl}_C$\,(3)}\BinaryInfC{$\SET{X}=\SET{x} \cf \psi$}
\DisplayProof
\\


\AxiomC{}\noLine\UnaryInfC{}\noLine\UnaryInfC{}\noLine\UnaryInfC{$\SET{X}=\SET{x}\cf \varphi$~~} \AxiomC{}\noLine\UnaryInfC{~~$\SET{X}=\SET{x}\cf \psi$} \RightLabel{\footnotesize$\cf\!\!\wedge\textsf{I}$} \BinaryInfC{$\SET{X}=\SET{x}\cf \varphi\land \psi$}\noLine\UnaryInfC{} \DisplayProof
\quad\quad\quad\AxiomC{}\noLine\UnaryInfC{}\noLine\UnaryInfC{}\noLine\UnaryInfC{$\neg(\SET{X}=\SET{x}\cf \alpha)$}\RightLabel{\footnotesize$\neg\!\cf\!\textsf{E}$ }\UnaryInfC{$\SET{X}=\SET{x}\cf \neg\alpha$} \noLine\UnaryInfC{}\DisplayProof
\\ 
~~\AxiomC{}\noLine\UnaryInfC{$\SET{X}=\SET{x} \cf (\SET{Y}=\SET{y} \cf \varphi)$}\RightLabel{\footnotesize$\cf\!\textsf{Extr}$ (4)~~}\UnaryInfC{$(\SET{X'}=\SET{x'} \land \SET{Y}=\SET{y}) \cf \varphi$}\DisplayProof
\AxiomC{$(\SET{X}=\SET{x} \land \SET{Y}=\SET{y}) \cf \varphi$}\RightLabel{\footnotesize$\cf\!\textsf{Exp}$ (5)}\UnaryInfC{$\SET{X}=\SET{x} \cf (\SET{Y}=\SET{y} \cf \varphi)$} \DisplayProof\!\!\!\!\!\!\!\\

\AxiomC{}\noLine\UnaryInfC{}\noLine\UnaryInfC{}\noLine\UnaryInfC{$X_1 \leadsto X_2  \hspace{10pt} \dots \hspace{10pt} \dots  \hspace{10pt} X_{k-1} \leadsto X_k$}\RightLabel{\footnotesize$\textsf{Recur}$\,}
\UnaryInfC{$\neg(X_k \leadsto X_1)$} \noLine\UnaryInfC{}\DisplayProof
\\

 
\multicolumn{1}{|L{1\linewidth}|}{{\footnotesize (1) $\theta$ is $\cf$-free. (2) $\SET X=\SET x$ is consistent. 
(3) The derivations $D_1,D_2,D$ do not have any undischarged assumptions. 
(4) $\SET X=\SET x$ is consistent, $\SET{X'} = \SET X \setminus \SET Y$, $\SET{x'} = \SET x \setminus \SET y$. 
(5)  $\SET{X}\cap \SET Y = \emptyset$. 
}
}\\
\hline
\end{tabular}
\end{center}
}


\end{itemize}

\end{df}







Note that the above system is parametrized with the signature $\sigma$. 
We write $\Gamma\ded\varphi$ (or simply $\Gamma\vdash\varphi$ when $\sigma$ is clear from the context) if the formula $\varphi$ can be derived from $\Gamma$ by applying the rules in  the above system. 

The connective rules are as in \cite{YanVaa2016}. The axioms \textsf{ValDef} and \textsf{ValUnq} are special cases of the \emph{Definiteness} and \emph{Uniqueness} axioms of Galles and Pearl (\cite{GalPea1998}). The rules $\cf$\textsf{Eff} 
and \textsf{Recur} correspond to the \emph{Effectiveness}
and \emph{Recursivity} from \cite{GalPea1998}. We remark that \textsf{Recur} is actually a scheme of rules (one rule for each $1<k\in\mathbb{\mathbb{N}}$. 
Its formulation strongly depends on the finitude of the signature, but it could be in principle replaced by a more elementary rule that does not have this requirement.\footnote{The possibility of expressing recursivity with infinite signatures is suggested in a footnote of \cite{Hal2000} and implemented as a rule in \cite{Bri2012}.
}
Furthermore, the role of $\leadsto$ in the formulation of  \textsf{Recur} could be taken instead by the relation of \emph{direct cause}, as shown in \cite{BarSchSmeVelXie2020}. 
Rule $\cf\textsf{I}$ is a rule of \emph{conjunction conditionalization}, from which the \emph{Composition} rule of \cite{GalPea1998} is derivable, as we show in the next proposition.  The rules $\cf\!\!\land$\textsf{I}, $\cf\!\!\lor$\textsf{E} and $\neg\!\cf\!\textsf{E}$ correspond to (part of) some additional rules introduced by Halpern (\cite{Hal2016}) for the purpose of reducing the complexity of consequents of unnested counterfactuals. Since we also allow right-nested counterfactuals, we need further rules $\cf$\textsf{Extr} (\emph{Extraction}) 
and $\cf$\textsf{Exp}   (\emph{Exportation}) to remove or add occurrences of $\cf$ from the consequents. 
A comment might be due concerning the notation used in $\cf$\textsf{Extr}. As $\SET X=\SET x$ abbreviates a conjunction of equalities $X_1=x_1 \land \dots \land X_n=x_n$, in the context of this rule $\SET X'=\SET x'$ denotes the formula obtained from $\SET X=\SET x$ after removing all conjuncts $X_i=x_i$ such that the variable $X_i$ is in $\SET Y$. The four remaining rules,  $\cf$\textsf{Rpl}$_A$, $\cf$\textsf{Rpl}$_C$, $\textsf{ex falso}_{\cf}$ and $\cf\!\bot\textsf{E}$, were recognized in \cite{BarSan2019} as implicitly used in earlier completeness proofs; they concern replacement in  antecedents and consequents of counterfactuals, and the treatment of inconsistent antecedents/consequents. We remark that a few of the rules in Definition \ref{co-system-df} and derived rules in Proposition \ref{derivable_rules_co} below are formulated using  the initial Greek letters $\alpha,\beta$;
we do this to emphasize that these rules are still sound for the extended languages $\COd$ and $\COv$ only when $\alpha,\beta$ are restricted to $\Co$-formulas. 

It is easy to verify that all rules in our system are sound over recursive  (generalized) causal teams. 
%
%
%
%
The axioms and rules in the deduction system of \cite{BarSan2019} are either included or derivable in our natural deduction system, as shown in the next proposition.



\begin{prp}\label{derivable_rules_co}
The following are derivable in the system for $\CO$: 
\begin{enumerate}
\item[(i)]\label{derivable_rules_co_mp}  $\alpha,\neg\alpha\vee\beta\vdash \beta$ (weak modus ponens) 
\item[(ii)]\label{derivable_rules_co_unq} $\SET{X}=\SET{x}\cf Y=y\vdash \SET{X}=\SET{x}\cf Y\neq y'$ (Uniqueness)
\item[(iii)]\label{derivable_rules_cf_cmp} $\SET X=\SET x\cf W=w,\SET X=\SET x\cf \theta\vdash(\SET X= \SET x\wedge W=w)\cf \theta$ whenever $\theta$ is $\cf$-free (Composition)
\item[(iv)]\label{derivable_rules_co_extr_conj} $\SET{X}=\SET{x}\cf \varphi\land \psi\vdash \SET{X}=\SET{x}\cf \varphi$ (Extraction of conjuncts)
\item[(v)]\label{derivable_rules_co_extr_neg} If $\SET X = \SET x$ is consistent, $\SET{X}=\SET{x}\cf \neg\alpha \vdash \neg(\SET{X}=\SET{x}\cf \alpha)$ 
\item[(vi)]\label{derivable_rules_cf_v_Dst} $\SET{X}=\SET{x}\cf \varphi\vee \psi\dashv\vdash (\SET{X}=\SET{x}\cf \varphi)\vee(\SET{X}=\SET{x}\cf \psi)$


\item[(vii)]\label{derivable_rules_cf_df} $\bigvee_{y\in \Ran(Y)} (\SET{X}=\SET{x}\cf Y=y)$ (Definiteness) 
\end{enumerate}
\end{prp}

\begin{proof}
Item (i
) follows  from $\neg\textsf{E}$ and $\vee\textsf{E}$. 
Items (ii
),(iv
) follow from $\textsf{ValUnq}$,  $\wedge\textsf{E}$ and $\cf\!\textsf{\corrb{Rpl}}_C$. 
For item (v
), we first derive by applying $\cf\!\wedge\textsf{I}$, and $\cf\!\bot\textsf{E}$  that
\begin{center}$
\SET{X}=\SET{x}\cf \neg\alpha, \SET{X}=\SET{x}\cf \alpha\vdash \SET{X}=\SET{x}\cf \neg\alpha\wedge\alpha\vdash\SET{X}=\SET{x}\cf\bot  
\vdash \bot
$\end{center}
Then, by $\neg\textsf{I}$ we conclude that $\SET{X}=\SET{x}\cf \neg\alpha\vdash \neg(\SET{X}=\SET{x}\cf \alpha)$.

For item (iii
), if $W\in \SET X$, then we derive $\SET X=\SET x\cf \theta\vdash(\SET X= \SET x\wedge W=w)\cf \theta$ by $\cf\textsf{\corrb{Rpl}}_A$ (if $W=w$ is among $\SET X=\SET x$) or $\textsf{ex falso}_{\cf}$
(if $W=w'$ is among $\SET X=\SET x$ for some $w'\neq w$). Now assume $W\notin \SET X$. Then we have that
\begin{align*}
&\SET X=\SET x\cf W=w,\SET X=\SET x\cf\theta\\
\vdash~& \SET X=\SET x\cf (W=w\wedge \theta)\tag{by $\cf\wedge\textsf{I}$}\\
\vdash~& \SET X=\SET x\cf (W=w\cf \theta)\tag{by $\land\textsf{E},\cf\textsf{I}$ and $\cf\textsf{\corrb{Rpl}}_C$}\\
\vdash~& (\SET X=\SET x\wedge W=w)\cf \theta\tag{by  $\cf\textsf{Extr}$}.
\end{align*}

For item (vi
), we first consider the case that $\SET X =\SET x$ is inconsistent. From right to left, we then have $\vdash \SET{X}=\SET{x}\cf \varphi\vee \psi$ by rule $\textsf{ex falso}_{\cf}$. From left to right, we can first derive $\SET{X}=\SET{x}\cf \varphi$ by  $\textsf{ex falso}_{\cf}$ and then $\vdash (\SET{X}=\SET{x}\cf \varphi)\vee (\SET{X}=\SET{x}\cf\psi)$ by $\vee\textsf{I}$. Now assume that $\SET X =\SET x$ is consistent.

For the right-to-left direction, first we derive 
\[\SET{X}=\SET{x}\cf\varphi\vdash \SET{X}=\SET{x}\cf\varphi\vee\psi\text{ and }\SET{X}=\SET{x}\cf\psi\vdash \SET{X}=\SET{x}\cf\varphi\vee\psi\] 
by applying $\cf \corrb{\textsf{Rpl}}_C$. Thus, by $\vee\textsf{E}$ we conclude that  $(\SET{X}=\SET{x}\cf\varphi)\vee(\SET{X}=\SET{x}\cf\psi)\vdash$ $\vdash\SET{X}=\SET{x}\cf\varphi\vee\psi$. 


For the left-to-right direction, we first use De Morgan's laws (which are easily derivable in the usual way) plus $\cf\!\textsf{\corrb{Rpl}}_C$ to derive from $\SET X =\SET x\cf \varphi \lor \psi$ that $\SET X =\SET x\cf \neg(\neg\varphi \land \neg \psi)$, which then gives, by item (v), that  $\neg\big(\SET X =\SET x\cf (\neg\varphi \land \neg \psi)\big)$. Next, we derive by applying $\neg\!\cf\!\textsf{E}$ and $\cf\!\!\wedge\textsf{I}$, together with $\wedge\textsf{E}$ and $\wedge\textsf{I}$, that
\[\neg(\SET X =\SET x\cf \varphi)\wedge \neg(\SET X =\SET x\cf \psi)\vdash(\SET X =\SET x\cf \neg\varphi)\wedge (\SET X =\SET x\cf \neg\psi)\vdash\SET X =\SET x\cf (\neg\varphi \land \neg \psi).\]
Then, by applying $\neg\textsf{I}$ and $\neg\textsf{E}$, we derive the contrapositive of the above, namely  
\[\neg\big(\SET X =\SET x\cf (\neg\varphi \land \neg \psi)\big)\vdash\neg\big(\neg(\SET X =\SET x\cf \varphi)\wedge \neg(\SET X =\SET x\cf \psi)\big).\]
By De Morgan's laws again, the formula in the conclusion above is equivalent to $(\SET{X}=\SET{x}\cf \varphi)\vee(\SET{X}=\SET{x}\cf \psi)$. This finishes the proof.

For item (vii
), we first derive by $\cf\!\textsf{Eff}$ that $\vdash \SET{X}=\SET{x}\cf X=x$, where  $X=x$ is an arbitrary equality from $\SET{X}=\SET{x}$. By \textsf{ValDef} we also have that $\vdash\bigvee_{y\in \Ran(Y)}Y=y$. Thus, we conclude by applying $\cf\!\!\textsf{Rpl}_C$ that $\vdash\SET{X}=\SET{x}\cf \bigvee_{y\in \Ran(Y)}Y=y$, which then implies that $\vdash\bigvee_{y\in \Ran(Y)} (\SET{X}=\SET{x}\cf Y=y)$ by \corrb{item (vi).} 
\end{proof}

\begin{teo}[Completeness]\label{completeness_co}
For any set $\Delta\cup\{\alpha\}$ of $\CO$-formulas, we have that 
\(\Delta\models^{c/g}\alpha\iff\Delta\vdash\alpha.\)
\end{teo}
\begin{proof}
Since our system derives all axioms and rules of the deduction system of \cite{BarSan2019}, the completeness of our system over causal teams follows from that of \cite{BarSan2019}.
The completeness of the system over generalized causal teams follows from the fact that $\Delta\models^{c}\alpha$ iff $\Delta\models^{g}\alpha$, given by Corollary \ref{gct_sem_cons_2_ct}.
\end{proof}


\section{Extensions of $\Co$} \label{SECCOMPLETE}

In this section, we study two extensions of \Co, namely $\COv$ and $\COd$. We first show in Section \ref{TEOCHARCDU} that the two logics are both expressively complete in the class of nonempty properties that are causally downward closed and closed under $\approx$-equivalence.  Then in Sections \ref{SUBSGENCOV} and \ref{sec:axiom_cod_g} we axiomatize the two logics over generalized causal team semantics. The axiomatizations for the logics over causal team semantics are given in Section \ref{AXCOCT}. 

\subsection{Expressive power of $\COv$ and $\COd$}

In this subsection, we characterize the expressive power of $\COv$ and $\COd$ over (generalized) causal teams. We show  that both logics characterize the  nonempty causally downward closed   team properties up to causal equivalence, and the two logics are thus expressively equivalent. 

\begin{teo}\label{TEOCHARCODGCT}\label{TEOCHARCDU}
Let  $\K$ be a nonempty class of (generalized) causal teams over some signature $\sigma$. Then the following are equivalent:
\begin{enumerate}
\item[(i)] $\K$ is  causally downward closed and closed under equivalence (or in other words, $\K$ is closed under $\succcurlyeq$). 
\item[(ii)] $\K$ is definable by a $\COV$-formula.
\item[(iii)] $\K$ is definable by a $\COD$-formula.
\end{enumerate}
\end{teo}

Notice that any such class $\K$ is finite. By Theorems \ref{TEOGENDW} and \ref{PROPEQUIV},  for every $\COV$- or $\COD$-formula $\varphi$, the set $\K_\varphi$ is nonempty, causally downward closed and closed under causal equivalence. Thus items (ii) and (iii) of the above theorem imply item (i). 
Since dependence atoms $\dep{\SET X}{Y}$ are definable in $\COV$ (see (\ref{dep_atm_dfnable_vvee}) and (\ref{con_atm_dfnable_vvee}) in Section \ref{SUBSCT}),  item (iii)  implies item (ii). It then suffices to show that item (i) implies item (iii). We give the proof for the causal team and generalized causal team case simultaneously.  In this proof, we make essential use of a formula $\Xi^T$ that defines the complement of the set of $\preccurlyeq$-successors of a (generalized) causal team $T$.  This formula resembles, in the causal setting, a similar one introduced in \cite{YanVaa2016}  in the pure team setting. 

The formula $\Xi^T$ is a disjunction of three parts. The first disjunct requires some words of comment. Its purpose is to characterize the size 
of a (generalized) causal team. 
The 
size 
of a causal team $T=(T^-,\F)$ can be (naturally) defined as the cardinality of its team component $T^-$. Generalized causal teams have a more complex structure, therefore the notion of size 
is more subtle. Consider for example the generalized causal team $T = \{(s,\F),(s,\G)\}$. If $\F$ and $\G$ are not similar (i.e., $\F\not\sim\G$), the two elements of $T$ should be counted as distinct; therefore, the cardinality of $T^-$ (i.e., $|T^-|=|\{s\}|=1$) is not a correct measure of the size of $T$. On the other hand, 
if $\F$ and $\G$ are similar (i.e., $\F\sim\G$, or $(s,\F)\approx(s,\G)$), the two elements $(s,\F)$ and $(s,\G)$ cannot be told apart using our languages; thus our languages do not allow us to distinguish the size of $T$ from that of one of its singleton subteams. We are then forced to say that $T$ has size $1$; therefore also $|T|=2$ is not a correct measure of the size of $T$. 
For these reasons, for generalized causal teams $T$, we shall be more concerned with the cardinality 
of the quotient team of $T$ over $\approx$ between pairs $(s,\F)\in T$, denoted as $T/_\approx$.

If $T$ is uniform, then clearly $|T/_{\approx}|=|T^-|$.  In particular, if $S=(S^-,\F)$ is a causal team, then for the generalized causal team $S^g=\{(s,\F)\mid s\in S^-\}$ it generates, we have that $|S^-|=|S^g/_\approx|$.

In the next lemma, we introduce a formula $\chi$ that defines the property of the quotient team $T/_\approx$ of having cardinality at most $1$.

\begin{lm}\label{phi_u_lm}
Let $T$ be a generalized causal team  over some vocabulary $\sigma$. 
\begin{enumerate}
    \item[(i)] If $|T^-|=1$, then   $T\models\mu\iff$ $T$ is uniform , where  
\[\mu:=\bigwedge_{V\in \Dom}\bigwedge_{\SET w\in\Ran(\SET W_V)} (\SET W_V = \SET w \cf \con V)~\text{ and }~\SET W_V = Dom\setminus \{V\}.\]
\item[(ii)] $T\models \chi \iff |T/_\approx|\leq 1$, where 
\[\chi:=\mu\land\bigwedge_{V\in \Dom}\con{V}.\]
\end{enumerate}
\end{lm}

\begin{proof}
(i) Let $T=\{(s,\F_1),\dots,(s,\F_n)\}$.
If $\F_1\sim\dots\sim\F_n$, then $T\approx\{(s,\F_1)\}$. Clearly, $\{(s,\F_1)\}\models\mu$, thus we obtain $T\models\mu$ by Theorem \ref{PROPEQUIV}.

For the converse direction, suppose $T$ is not uniform, i.e., $\F_i\not\sim\F_j$ for some $1\leq i<j\leq n$. If (w.l.o.g.) $\End(\F_i)\setminus \Con(\F_i)\nsubseteq \End(\F_j)\setminus \Con(\F_j)$, then there exists $V$ such that $V\in \End(\F_i)\setminus \Con(\F_i)$ and $V\notin \End(\F_j)\setminus \Con(\F_j)$. The former implies that 
there are $\SET p,\SET p'\in \Ran(PA_V^{\F_i})$ such that $(\F_i)_V(\SET p) \neq (\F_i)_V(\SET p')$. Let $\SET w,\SET w'\in \Ran(\SET W_V)$ be extensions of the sequences $\SET p,\SET p'$ respectively. Clearly, $s^{\F_i}_{\SET W_V = \SET w} \neq s^{\F_i}_{\SET W_V = \SET w'}$. Meanwhile, the second assumption implies  that $s^{\mathcal F_j}_{\SET W_V = \SET w} = s^{\mathcal F_j}_{\SET W_V = \SET w'}$. Therefore, either $s^{\mathcal F_j}_{\SET W_V = \SET w}(V)= s^{\mathcal F_j}_{\SET W_V = \SET w'}(V)\neq s^{\F_i}_{\SET W_V = \SET w}(V)$ or $s^{\mathcal F_j}_{\SET W_V = \SET w}(V)= s^{\mathcal F_j}_{\SET W_V = \SET w'}(V)\neq s^{\F_i}_{\SET W_V = \SET w'}(V)$.  But then $T\not\models W_V = \SET w\cf \con V$, which implies $T\not\models\mu$.

In the second case we have that, for some $V\in\End(\F_i)\setminus \Con(\F_i)=\End(\mathcal F_j)\setminus \Con(\mathcal F_j)$, there are  $\SET x\in \Ran(PA_V^{\F_i}\cap PA_V^{\mathcal F_j})$, $\SET y\in \Ran(PA_V^{\F_i}\setminus PA_V^{\mathcal F_j})$ and  $\SET z \in \Ran(PA_V^{\mathcal F_j}\setminus PA_V^{\mathcal F_i})$ such that  $(\F_i)_V(\SET{xy})\neq (\mathcal F_j)_V(\SET{xz})$. Then, for any extension $\SET w \in \Ran(\SET W_V)$  of the sequence $ \SET{xyz}$, we have $s^{\F_i}_{\SET W_V = \SET w}(V) \neq s^{\mathcal F_j}_{\SET W_V = \SET w}(V)$, which again implies $T\not\models\mu$.

(ii) It is easy to see that  $T\models\bigwedge_{V\in \Dom}\con{V}$ iff $|T^-|\leq 1$. If $T=\emptyset$, then $\emptyset\models\chi$. If $T\neq\emptyset$, then by item (i), we have that 
\[T\models\chi \iff |T^-|= 1\text{ and $T$ is uniform} \iff |T/_\approx|=|T^-|=1.\] 
\end{proof}

 A causal team $T$ can be identified with a (uniform) generalized causal team $T^g$, and, as pointed out already, $|T^-|=|T^g/_\approx|$. The above lemma thus gives a characterization of ``having size/cardinality 
 at most $1$" for causal teams as well. Furthermore, since $T^g$ is always uniform, the formula $\mu$  that defines uniformity in item (i) above becomes trivial  over singleton causal teams. \corrb{Therefore:}

\begin{coro}
\corrb{For causal teams the formula $\chi$ defined above can be simplified as follows:}
\[
 \displaystyle\chi\equiv^c\bigwedge_{V\in \Dom}\con{V}.
\]
\end{coro}
\begin{proof}
For any causal team $S$, we have that 
\begin{align*}
    S\models^c\chi&\iff S^g\models^g\chi\tag{by Lemma \ref{LEMIDENTIFY}(i)}\\
    &\iff|S^g/_\approx|\leq 1\tag{by Lemma \ref{phi_u_lm}}
    \\
    &\iff |(S^g)^-|\leq 1
    \iff |S^-|\leq 1
    \iff S\models^c\bigwedge_{V\in \Dom}\con{V}.
\end{align*}
\end{proof}

We now introduce \COD-formulas that characterize the cardinality of $T/_{\approx}$. For every natural number $k$, define a formula $\chi_k$ as
\[\chi_0=\bot~~\text{ and }~~\chi_{k}:=\underbrace{\chi\lor\cdots\lor \chi}_{k}\text{ for }k\geq1\]


\begin{prp}\label{LEM_chi_prop}Let $\sigma$ be a vocabulary and $k$ a natural number.
\begin{enumerate}
    \item[(i)] For any generalized causal team $T$ over $\sigma$,
$T\models^g\chi_k \iff |T/_\approx|\leq k$.
\item[(ii)] For any causal team $T$ over $\sigma$,
$T\models^c\chi_k \iff |T^-|\leq k$.
\end{enumerate}
\end{prp}
\begin{proof}
(i) Let $T$ be a generalized causal team. Clearly, $T\models \chi_0$ iff $T=\emptyset$ iff $|T/_{\approx}|=0$. For any $k\geq 1$, $T\models\chi_k$ iff there are $T_1,\dots,T_k$ with  $T=T_1\cup\dots\cup T_k$ and each $T_i\models\chi$ (i.e., by Lemma \ref{phi_u_lm},  $|T_i/_\approx|\leq 1$), iff there are pairwise non-$\approx$-equivalent teams $T_{i_1},\dots,T_{i_m}$ with $m\leq k$, $T=T_{i_1}\cup\dots\cup T_{i_m}$ and each $|T_{i_j}/_\approx|\leq 1$, iff $|T/_\approx|=m\leq k$ for some $m$.

(ii) For any causal team $T$,  by item (i) we have that
\begin{align*}
    T\models^c\chi_k&\iff T^g\models^g\chi_k\tag{by Lemma \ref{LEMIDENTIFY}(i)}\\
    &\iff |T^-|=|(T^g)^-|=|T^g/_\approx|\leq k\tag{since $T^g$ is uniform}
\end{align*}
\end{proof}

Now, we are ready to define the crucial formula $\Xi^T$ in \COD.
Let $T$ be a nonempty generalized causal team with $|T/_\approx|= k+1$. Define 
\[\Xi^T:=\chi_k\vee\Theta^{\overline{T^-}}\vee\mathop{\bigvee_{s\in T^-\hspace{-2pt},\hspace{2pt}\F\in\FUN}}_{\{(s,\F)\}\not\preccurlyeq T}(\Theta^{\{s\}}\wedge\Phi^\F),\]
 where $\overline{T^-}=\ASS\setminus T^-$. 
If $T$ 
is a causal team, define $\Xi^T:=\Xi^{T^g}$. 
We show in the next lemma 
the desired characterization property of the formula $\Xi^T$. Recall from Lemma \ref{preccurlyeq_prop}(i) 
that $T\not\preccurlyeq S$ iff there exists $(t,\F)\in T$ such that $\{(t,\F)\}\not\preccurlyeq S$. 


\begin{lm}\label{LEMNOTSUB}
Let $S,T$ be two (generalized) causal teams over some signature $\sigma$ with $T\neq\emptyset$. 
Then, $S\models \Xi^T \iff T\not\preccurlyeq S$.
\end{lm}
\begin{proof}
We first show the generalized causal team case. ``$\Longrightarrow$": Suppose $S\models\Xi^T$. Then there exist $S_0,S_1,S_2\subseteq S$ such that $S=S_0\cup S_1\cup S_2$, $S_0\models\chi_k$, $S_1\models\Theta^{\overline{T^-}}$ and
\begin{equation}\label{LEMNOTSUB_eq}
    S_2\models \mathop{\bigvee_{s\in T^-\hspace{-2pt},\hspace{2pt} \F\in\FUN}}_{\{(s,\F)\}\not\preccurlyeq T}(\Theta^{\{s\}}\wedge\Phi^\F).
\end{equation}
The first clause implies that
$|S_0/_\approx|\leq k$, by Proposition \ref{LEM_chi_prop}. Since $|T/_\approx|=k+1$, there must exist a $(t,\G)\in T$ such that $\{(t,\G)\}\not\preccurlyeq S_0$. Moreover, since  $S_1\models\Theta^{\overline{T^-}}$, we have by Lemma \ref{LEMCHARSUB_ct} that $ S_1^-\subseteq \overline{T^-}$. It then follows that $t\notin S_1^-$, giving that $\{(t,\G)\}\not\preccurlyeq S_1$ as well. Finally, we show that $\{(t,\G)\}\not\preccurlyeq S_2$, which would then imply $\{(t,\G)\}\not\preccurlyeq S_0\cup S_1\cup S_2=S$, thereby $T\not\preccurlyeq S$.

Suppose towards a contradiction that $\{(t,\G)\}\preccurlyeq S_2$. Then $(t,\mathcal{H})\in S_2$ for some $\mathcal{H}\sim\G$. Now, since (\ref{LEMNOTSUB_eq}) holds, we have that $\{(t,\mathcal{H})\}\models \Theta^{\{s\}}\wedge\Phi^{\F}$ for some $s\in T^-$ and $\F\in \FUN$ with $\{(s,\F)\}\not\preccurlyeq T$. This means, by Lemma \ref{LEMCHARSUB_ct} and Theorem \ref{LEMPHIF},  that $t=s\in T^-$ and $\F\sim\mathcal{H}\sim \G$. But since $(t,\G)\in T$, we must have that $\{(s,\F)\}\preccurlyeq T$, which is a contradiction. Hence, we conclude that $\{(t,\G)\}\not\preccurlyeq S_2$.


``$\Longleftarrow$": Suppose $T\not\preccurlyeq S$. Let 
\[S_0=\{(s,\F)\in S\mid \{(s,\F)\}\preccurlyeq T\}.\]
Then obviously $|S_0/_\approx| \leq |T/_\approx|$. Since $T\not\preccurlyeq S$, we must have that $T\not\preccurlyeq S_0\subseteq S$. Thus, there exists $(s,\G)\in T$ such that $\{(s,\G)\}\not\preccurlyeq S_0$. 
But then $|S_0/_\approx|<|T/_\approx|=k+1$, i.e., $|S_0/_\approx|\leq k$. By Proposition \ref{LEM_chi_prop}, we conclude that $S_0\models\chi_k$.

Let 
\[S_1=\{(s,\F)\in S\mid s\notin T^-\}.\]
Since $S_1^-\subseteq \overline{T^-}$,  by Lemma \ref{LEMCHARSUB_ct} we obtain $S_1\models\Theta^{\overline{T^-}}$. 

Let $S_2=S\setminus (S_0\cup S_1)$. Observe that
\[S_2=\{(s,\F)\in S\mid s\in T^-\text{ and }\{(s,\F)\}\not\preccurlyeq T\}.\]
For each $(s,\F)\in S_2$, we have that $\{(s,\F)\}\models \Theta^{\{s\}}\wedge\Phi^\F$ by Lemma \ref{LEMCHARSUB_ct} and Theorem \ref{LEMPHIF}. 
Thus (using also the empty team property) \eqref{LEMNOTSUB_eq} holds. Altogether, we conclude that $S\models\varphi$.

\vspace{4pt}

Next, observe that the causal team case follows from the generalized causal team case, since for any causal teams $S,T$, we have that $\Xi^T=\Xi^{T^g}$ and
\[S\models^c\Xi^T\iff S^g\models^g\Xi^{T^g}\iff T^g\not\preccurlyeq S^g\iff T\not\preccurlyeq S\]
by Lemma \ref{LEMIDENTIFY}(i) and Lemma \ref{preccurlyeq_prop}(iv).
\end{proof}

Finally, we give the proof of our main theorem of this section.

\vspace{5pt}

\begin{lateproof}{Theorem \ref{TEOCHARCODGCT}}
We prove that item (i) implies item (iii). Let $\K$ be a nonempty class of (generalized) causal teams as described in item (i). Since $\K$ is nonempty and causally downward closed, the 
empty team belongs 
to $\K$. Thus, every (generalized) causal team $T\in \CT\setminus\K$ is nonempty, where $\CT$ denotes the (finite) set of all (generalized) causal teams over  $\sigma$. 
Now, define 
\(\displaystyle\varphi=\bigwedge_{T\in \CT \setminus \K}\Xi^T.\)
We show that $\K=\K_\varphi$. 

For any $S\notin \K$, i.e., $S\in \CT \setminus \K$, since $S\preccurlyeq S$ and $S^-\neq\emptyset$, by Lemma \ref{LEMNOTSUB} we have that $S\not\models \Xi^S$. Thus $S\not\models\varphi$, i.e., $S\notin \K_\varphi$.
Conversely, suppose $S\in \K$. Take any $T\in \CT \setminus \K$. If $T\preccurlyeq S$, then since $\K$ is closed under $\succcurlyeq$ 
we must conclude that $T\in \K$;  a contradiction. Thus, by Lemma \ref{LEMNOTSUB}, $S\models \Xi^T$. Hence $S\models\varphi$, i.e., $S\in \K_\varphi$.
\end{lateproof}


\subsection{Axiomatizing $\COv$  over generalized causal teams} \label{SUBSGENCOV}

In this section, 
we introduce a sound and complete system of natural deduction for $\COV$, which extends  the system  for $\CO$, and can also be seen as
a variant of the systems for propositional dependence logics introduced in \cite{YanVaa2016}.

\begin{df}\label{COV_system_df}
The system of natural deduction for $\COV$ 
over generalized causal teams 
consists of all rules of the system for $\CO$ (see Definition \ref{co-system-df}) together with the following rules, where note that in the rules $\vee\textsf{E}$, $\neg\textsf{I}$, $\neg\textsf{E}$, $\textsf{RAA}$ and $\neg\! \ \!\!\cf\!\!\textsf{E}$ from Definition \ref{co-system-df} the formula $\alpha$ ranges over $\CO$-formulas only:
\begin{itemize}
\item[-] Additional rules for $\vee$:\\
{\normalfont
\renewcommand{\arraystretch}{1.8}
\begin{tabular}{|C{0.45\linewidth}C{0.45\linewidth}|}
\hline
 \multirow{2}{*}{\AxiomC{}\noLine\UnaryInfC{$\varphi\vee\psi$} \RightLabel{$\vee\textsf{Com}$}\UnaryInfC{$\psi\vee\varphi$}\DisplayProof
 \quad
 \AxiomC{}\noLine\UnaryInfC{$(\varphi\vee\psi)\vee\chi$} \RightLabel{$\vee\textsf{Ass}$}\UnaryInfC{$\varphi\vee(\psi\vee\chi)$}\DisplayProof}
 
 &
 
 \quad \AxiomC{$\varphi\vee\psi$}\AxiomC{}\noLine\UnaryInfC{[$\varphi$]}\noLine\UnaryInfC{$\vdots$}\noLine\UnaryInfC{$\chi$} \RightLabel{$\vee\textsf{Rpl}$} \BinaryInfC{$\chi\vee\psi$}\noLine\UnaryInfC{}\DisplayProof\\
   \multicolumn{2}{|c|}{
 \AxiomC{$\SET{X}=\SET{x}\cf \varphi\vee \psi$}\RightLabel{\footnotesize$\cf\!\!\vee\textsf{Dst}$}\doubleLine\UnaryInfC{$(\SET{X}=\SET{x}\cf \varphi)\vee(\SET{X}=\SET{x}\cf \psi)$}\noLine\UnaryInfC{} \DisplayProof 
 }\\
 \hline
 \end{tabular}
 %
}


\item[-] Rules for $\vvee$:\\
{\normalfont
\setlength{\tabcolsep}{4pt}
\renewcommand{\arraystretch}{1.8}
\setlength{\extrarowheight}{1pt}
\begin{tabular}{|C{0.45\linewidth}C{0.48\linewidth}|}
\hline
 \multirow{2}{*}{\AxiomC{$\varphi$} \RightLabel{$\vvee\!\textsf{I}$}\UnaryInfC{$\varphi\vvee\psi$}\DisplayProof
 ~\AxiomC{$\varphi$} \RightLabel{$\vvee\!\textsf{I}$}\UnaryInfC{$\psi\vvee\varphi$}\DisplayProof}
 
 &
 
 \AxiomC{$\varphi\vvee\psi$}\AxiomC{}\noLine\UnaryInfC{[$\varphi$]}\noLine\UnaryInfC{$\vdots$}\noLine\UnaryInfC{$\chi$} \AxiomC{}\noLine\UnaryInfC{[$\psi$]}\noLine\UnaryInfC{$\vdots$}\noLine\UnaryInfC{$\chi$}\RightLabel{$\vvee\!\textsf{E}$} \TrinaryInfC{$\chi$}\DisplayProof\\
 
 & \\
 
  \multicolumn{2}{|c|}{\AxiomC{$\varphi\vee(\psi\vvee\chi)$}\RightLabel{$\vee\!\vvee\!\textsf{Dst}$}\UnaryInfC{$(\varphi\vee\psi)\vvee(\varphi\vee\chi)$} \noLine\UnaryInfC{}\DisplayProof
 \AxiomC{$\SET X = \SET x \cf \psi \vvee \chi$}\RightLabel{$\cf\!\!\!\vvee$\!\textsf{Dst}}\UnaryInfC{$(\SET X = \SET x \cf \psi) \vvee (\SET X = \SET x \cf \chi)$} \noLine\UnaryInfC{}\DisplayProof}\\\hline
\end{tabular}
}
\end{itemize}

\vspace{3pt}

\noindent The double line in rule $\cf\!\!\vee\textsf{Dst}$ indicates that we are also including the inverse of the rule in the calculi.
\end{df}


\noindent The rules in our system are clearly sound. Since our system for \COV~ is an extension of that for \CO, the clauses in Proposition \ref{derivable_rules_co} are still derivable (assuming that the formula $\alpha$ that appears in some of the clauses ranges over \CO-formulas only) by the same derivations, except for item (vi) whose proof uses the rule $\vee\mathsf{E}$ and the rules for negation which are not sound for arbitrary \COV-formulas. 
This item is thus now directly included in the system for \COV as rule $\cf\!\!\vee\textsf{Dst}$. 
We now derive some additional useful clauses.


\begin{prp}\label{derivable_rules_covvee}
The following are derivable in the system of  $\COV$:
\begin{enumerate}
\item[(i)]\label{derivable_rules_covvee_com_ass}  $\varphi\vvee\psi\vdash \psi\vvee\varphi$ and $(\varphi\vvee\psi)\vvee\chi\vdash\varphi\vvee(\psi\vvee\chi)$
\item[(ii)]\label{derivable_rules_covvee_con_distr}  $\varphi\wedge(\psi\vvee\chi)\dashv\vdash (\varphi\wedge\psi)\vvee(\varphi\wedge\chi)$
\item[(iii)]\label{derivable_rules_cod_dstr}  $\varphi\wedge(\psi\vee\chi)\vdash(\varphi\wedge\psi)\vee(\varphi\wedge\chi)$ 
\item[(iv)]\label{derivable_rules_covvee_vvee_distr} $\varphi\vee(\psi\vvee\chi)\dashv\vdash(\varphi\vee\psi)\vvee(\varphi\vee\chi)$
\item[(v)]\label{derivable_rules_covvee_vvee_cf} $\SET X = \SET x \cf \psi\vvee\chi\dashv\vdash (\SET X = \SET x \cf \psi) \vvee (\SET X = \SET x \cf \chi)$
\end{enumerate}
\end{prp}
\begin{proof}
Items (i
),(ii
) follow from the standard argument using $\vvee\textsf{I}$ and $\vvee\textsf{E}$. Item (iii
) is proved easily by applying $\vee\textsf{Rpl}$ and $\lor$\textsf{Com}; we remark that it cannot be proved by the usual natural deduction arguments, as rule $\lor\textsf{E}$ is available only for $\CO$ formulas. 
For item (iv
), the left to right direction follows from $\vee\vvee\textsf{Dst}$. For the right to left direction, we first derive $\varphi\vee\psi\vdash\varphi\vee(\psi\vvee\chi)$ and $\varphi\vee\chi\vdash\varphi\vee(\psi\vvee\chi)$ by applying $\vvee\textsf{I}$ and $\vee\textsf{Rpl}$. Then $(\varphi\vee\psi)\vvee(\varphi\vee\chi)\vdash\varphi\vee(\psi\vvee\chi)$ follows from $\vvee\textsf{E}$. For item (v
), the left to right direction follows from $\cf\!\!\vvee$\textsf{Dst}, and the other direction is proved by using $\vvee\textsf{I}$,  $\cf\!\textsf{\corrb{Rpl}}_C$ and $\vvee\!\textsf{E}$.
\end{proof}

An important lemma towards the completeness theorem states that every $\COV$-formula $\varphi$ is provably equivalent to the $\vvee$-disjunction of a (finite) set of $\CO$-formulas. Formulas of this type are called \emph{resolutions} of $\varphi$ in \cite{CiaRoe2011}.


\begin{df}
Let $\varphi$ be a $\COV$-formula. Define the set $\mathcal R(\varphi)$ of its \textbf{resolutions} inductively as follows:
\begin{itemize}
\item[-] $\mathcal R(X=x) = \{X=x\}$,
\item[-] $\mathcal R(\neg\alpha) = \{\neg\alpha\}$, 
\item[-] $\mathcal R(\psi\land\chi) = \{\alpha\land\beta \ | \ \alpha \in \mathcal R(\psi), \beta \in \mathcal R(\chi) \}$,
\item[-] $\mathcal R(\psi\lor\chi) = \{\alpha\lor\beta \ | \ \alpha \in \mathcal R(\psi), \beta \in \mathcal R(\chi)\}$,
\item[-] $\mathcal R(\psi\vvee\chi) = \mathcal R(\psi)\cup\mathcal R(\chi)$,
\item[-] $\mathcal R(\SET X=\SET x\cf\varphi) = \{\SET X=\SET x\cf\alpha \ | \ \alpha \in \mathcal R(\varphi)\}$.
\end{itemize}
\end{df}

\noindent The set $\mathcal R(\varphi)$ is clearly a finite set of $\CO$-formulas. 

%

\begin{lm} \label{normal_form_derivable}
For any formula $\varphi\in \COV$, we have that 
\(\varphi\dashv\vdash\bigvvee \mathcal R(\varphi).\)
\end{lm}

\begin{proof}
We prove the lemma by induction on $\varphi$. 
If $\varphi$ is $X=x$ or $\neg\alpha$ for some  $\CO$-formula $\alpha$, then $\mathcal R(\varphi)=\{\varphi\}$, and 
$\varphi\dashv\vdash\bigvvee \mathcal R(\varphi)$ holds trivially. 



Now, suppose $\psi\dashv\vdash\bigvvee \mathcal R(\psi)$ and $\chi\dashv\vdash\bigvvee \mathcal R(\chi)$. If $\varphi=\psi\wedge\chi$, 
we derive by applying Proposition \ref{derivable_rules_covvee}(\ref{derivable_rules_covvee_con_distr}), $\land\textsf{I}$ and $\land\textsf{E}$
that 
%
\[\psi\wedge\chi\dashv\vdash \big(\bigvvee \mathcal R(\psi)\big)\wedge\big(\bigvvee \mathcal R(\chi)\big)\dashv\vdash \bigvvee\{\alpha\wedge \beta\mid \alpha\in \mathcal R(\psi),\beta\in \mathcal R(\chi)\}\dashv\vdash \bigvvee \mathcal R(\psi\wedge\chi).\]

If $\varphi=\psi\vee\chi$, we have analogous derivations using Proposition \ref{derivable_rules_covvee}(iv), $\lor\textsf{I}$ and $\lor\textsf{Rpl}$. 

If $\varphi=\psi\vvee\chi$, then by the rules $\vvee\textsf{I}$ and $\vvee\textsf{E}$ together with Proposition \ref{derivable_rules_covvee}(i)   
we have that
\[\psi\vvee\chi\dashv\vdash \big(\bigvvee \mathcal R(\psi)\big)\vvee\big(\bigvvee \mathcal R(\chi)\big)\dashv\vdash
\bigvvee \big(\mathcal R(\psi) \cup  \mathcal R(\chi)\big)
\dashv\vdash \bigvvee \mathcal R(\psi\vvee\chi).\]

If $\varphi=\SET X=\SET x\cf\psi$, then \allowdisplaybreaks
\begin{align*}
\SET X=\SET x\cf\psi&\dashv\vdash \SET X=\SET x\cf \bigvvee \mathcal R(\psi)\tag{$\cf\!\!\textsf{Rpl}_C$}
\\
&\dashv\vdash \bigvvee \{\SET X=\SET x\cf \alpha\mid \alpha\in \mathcal R(\psi)\} \tag{Proposition \ref{derivable_rules_covvee}(v)
}
\\
&\dashv\vdash \bigvvee \mathcal R(\SET X=\SET x\cf\psi).
\end{align*}
\end{proof}

\begin{teo}[Completeness]\label{TEOCOMPLCOU}
For any set  $\Gamma\cup\{\psi\}$ of $\COV$-formulas, we have that 
\(\Gamma\models^g\psi\iff\Gamma\vdash\psi.\)
\end{teo}
\begin{proof}
We prove the ``$\Longrightarrow$'' direction. Observe that there are only finitely many classes of causal teams of signature $\sigma$. Thus, any set of $\COV$-formulas is equivalent to a single $\COV$-formula, and it then suffices to prove the statement for $\Gamma = \{\varphi\}$.  

Now suppose $\varphi\models\psi$. Then by Lemma \ref{normal_form_derivable} 
and soundness we have that
\(
\bigvvee \mathcal R(\varphi) \models \bigvvee\mathcal R(\psi).
\)
Thus, for every $\gamma\in \mathcal R(\varphi)$, 
\(
\gamma \models^g \bigvvee\mathcal R(\psi),
\)
which further implies, 
by 
Lemma \ref{splitting_prop}, that there is an $\alpha_\gamma\in \mathcal R(\psi)$ such that $\gamma \models \alpha_\gamma$. Since $\gamma,\alpha_\gamma$ are $\CO$-formulas, and the system for $\COV$ extends that for $\CO$, we obtain by  the completeness theorem of $\CO$ (Theorem \ref{completeness_co})  that $\gamma \vdash \alpha_\gamma$. 
Applying $\vvee$\textsf{I} and Lemma \ref{normal_form_derivable}, we obtain
 \(
\gamma \vdash \bigvvee \mathcal R(\psi) \vdash \psi
\)
for each $\gamma\in \mathcal R(\varphi)$. Thus, by Lemma \ref{normal_form_derivable} and $\vvee$\textsf{E}, we conclude that
 \(
\varphi\vdash\bigvvee \mathcal R(\varphi) \vdash \psi.
\)
\end{proof}
\subsection{Axiomatizing $\COd$ over generalized causal teams}\label{sec:axiom_cod_g}



In this section we provide calculi for the language $\COD$, as interpreted over the generalized semantics. The resulting system is a variant of the deduction system for propositional dependence logic given in \cite{YanVaa2016}.


Given a subformula $\theta$  of $\varphi$, we use the 
notation $\varphi(\psi/[\theta,k])$ for the formula obtained by replacing \emph{the $k$th occurrence}  of the formula $\theta$ in $\varphi$ with $\psi$. For instance, if $\varphi=\,\,\depc{X}\vee(Y=y\cf\depc{X})$, then $\varphi(X=x/[\depc{X},2])=\,\,\depc{X}\vee(Y=y\cf X=x)$.

\begin{df}\label{system_cod_g}
The system for $\COD$ over generalized causal teams consists of all the rules of the system for $\CO$ (Definition \ref{co-system-df}) plus the additional rules for $\vee$ from Definition \ref{COV_system_df} and the following rules for dependence atoms:

\vspace{-10pt}

{\normalfont \begin{center}
\renewcommand{\arraystretch}{1.8}
\begin{tabular}{|C{0.4\linewidth}C{0.54\linewidth}|}
\hline
\AxiomC{}\noLine\UnaryInfC{}\noLine\UnaryInfC{}\noLine\UnaryInfC{$X=x$} \RightLabel{$\textsf{ConI}$}\UnaryInfC{$\depc{X}$}\noLine\UnaryInfC{}\noLine\UnaryInfC{}\noLine\UnaryInfC{}\noLine\UnaryInfC{}\DisplayProof 

&

\AxiomC{}\noLine\UnaryInfC{}\noLine\UnaryInfC{}\noLine\UnaryInfC{}\noLine\UnaryInfC{$\dep{X_1,\dots,X_n}{Y}$}\AxiomC{$\depc{X_1}\,\dots\, \depc{X_n}$} \RightLabel{$\textsf{DepE}$}\BinaryInfC{$\depc{Y}$}\noLine\UnaryInfC{}\noLine\UnaryInfC{}\noLine\UnaryInfC{}\noLine\UnaryInfC{}\DisplayProof
\\


\AxiomC{}\noLine\UnaryInfC{}\noLine\UnaryInfC{}\noLine\UnaryInfC{}\noLine\UnaryInfC{}\noLine\UnaryInfC{}\noLine\UnaryInfC{}\noLine\UnaryInfC{}\noLine\UnaryInfC{$\varphi$}\AxiomC{}\noLine\UnaryInfC{$\forall x\in Ran(X)$}\noLine\UnaryInfC{[$\varphi(X=x/[\depc{X},k])$]}\noLine\UnaryInfC{$\vdots$}\noLine\UnaryInfC{$\psi$}\RightLabel{$\textsf{ConE}$}\BinaryInfC{$\psi$}\noLine\UnaryInfC{} \DisplayProof

&

\AxiomC{}\noLine\UnaryInfC{[$\depc{X_1}$]} \AxiomC{}\noLine\UnaryInfC{$\dots$} \AxiomC{}\noLine\UnaryInfC{[$\depc{X_n}$]}\noLine\TrinaryInfC{}
\branchDeduce
\DeduceC{$\depc{Y}$}\RightLabel{$\mathsf{DepI}$}\UnaryInfC{$\dep{X_1,\dots,X_n}{Y}$}\DisplayProof
\\
\hline
\end{tabular}
\end{center}

\noindent Again, for the rules from Definition \ref{co-system-df} we assume that $\alpha$ ranges over $\CO$ formulas only. The rule $\textsf{ConE}$ states that if $\psi$ can be derived from the formula $\varphi(X=x/[\depc{X},k])$ for every $x\in \Ran(X)$, then we conclude $\psi$ from $\varphi$. The rule $\mathsf{DepI}$ is read as: if the constancy atom $\depc{Y}$ can be derived from the constancy atoms $\depc{X_1},\dots,\depc{X_n}$ (possibly together with other assumptions $\Gamma$), then the dependence atom $\dep{X_1,\dots,X_n}{Y}$ follows (from $\Gamma$).
}

\end{df}

\noindent Since the above system  contains all applicable rules from the system for $\COV$, Proposition \ref{derivable_rules_co} and Proposition \ref{derivable_rules_covvee}(iii) are still derivable in the system for $\COD$.
It is easy to show that the entailment relation $\vdash$ is monotone with respect to the substitution of positive subformula occurrences. 

\begin{lm}\label{monotonicity_derivations}
Let $\varphi,\theta,\theta'$ be $\COD$-formulas. Suppose $\theta$ is a subformula of $\varphi$ that is not in the scope of negation $\neg$, nor in the antecedent of a counterfactual implication $\cf$. Then, $\theta\vdash\theta'$ implies $\varphi\vdash\varphi(\theta'/[\theta,k])$.
\end{lm}
\begin{proof}
A routine inductive proof, using $\cf\!\!\textsf{Rpl}_A$ and $\cf\!\!\textsf{Rpl}_C$ in  the $\cf$ case.
\end{proof}

In the next proposition we derive some other useful clauses involving dependence atoms.

\begin{prop}\label{derivable_rules_cod}
\begin{enumerate}
\item[(i)]\label{derivable_rules_cod_1} $\dep{\SET X}{Y},\SET X=\SET x\vdash\depc{Y}$

\item[(ii)]\label{derivable_rules_cod_2} $\dep{\SET X}{Y}\dashv\vdash \bigvee_{\SET x \in Ran(\SET X)} (\SET X = \SET x \ \wedge  \depc{Y})$ 
\end{enumerate}
\end{prop}
\begin{proof}
Item (i) follows from  $\textsf{ConI}$ and $\textsf{DepE}$. We now prove item (ii). We will implicitly use Lemma \ref{monotonicity_derivations} in many of the steps. For the left to right direction, we first have by $\textsf{ValDef}$  that $\vdash \bigvee_{\SET x \in Ran(\SET X)} \SET X = \SET x $. Thus, 
\begin{align*}
\dep{\SET X}{Y}&\vdash\, \dep{\SET X}{Y}\wedge (\bigvee_{\SET x \in Ran(\SET X)} \SET X = \SET x)\tag{$\land\textsf{I}$}\\
&\vdash \bigvee_{\SET x \in Ran(\SET X)} (\SET X = \SET x\wedge  \dep{\SET X}{Y})\tag{Proposition \ref{derivable_rules_covvee}(iii), $\lor$\textsf{Com}}\\
&\vdash \bigvee_{\SET x \in Ran(\SET X)} (\SET X = \SET x\wedge  \depc{Y})\tag{item (i), $\vee\textsf{Rpl}$}.
\end{align*}
For the right to left direction, putting $\SET X=\langle X_1,\dots,X_n\rangle$, by $\textsf{DepI}$, it suffices to derive $\depc{X_1},\dots,\depc{X_n}, \bigvee_{\SET x \in Ran(\SET X)} (\SET X = \SET x\wedge  \depc{Y})\vdash \depc{Y}$, which, by $\textsf{ConE}$, is further reduced to deriving $\SET X=\SET x', \bigvee_{\SET x \in Ran(\SET X)} (\SET X = \SET x\wedge  Y=y)\vdash \depc{Y}$ for all $\SET x'\in \Ran(\SET X)$ and all $y\in \Ran(Y)$. Now, we have that
\begin{align*}
    \SET X=\SET x', \bigvee_{\SET x \in Ran(\SET X)} (\SET X = \SET x\wedge  Y=y)&\vdash \bigvee_{\SET x' \in Ran(\SET X)} (\SET X = \SET x\wedge \SET X = \SET x'\wedge  Y=y)\tag{$\land\textsf{I}$, Proposition \ref{derivable_rules_covvee}(iii)}\\
    &\vdash  (\SET X = \SET x\wedge  Y=y)\vee\bigvee_{\substack{\SET x'\neq \SET x \\ \SET x' \in Ran(\SET X)}} \bot \tag{\textsf{ValUnq}}\\
   &\vdash  \SET X = \SET x\wedge  Y=y \tag{$\neg\textsf{E}$, $\lor\textsf{E}$}\\
&\vdash  \depc{Y}\tag{$\land\textsf{E}$, $\textsf{ConI}$}.
\end{align*}
\end{proof}

\noindent Since dependence atoms never occur in the scope of $\neg$ or in the antecedent of $\cf$, by Lemma \ref{monotonicity_derivations} and Proposition \ref{derivable_rules_cod}(ii), we can transform every $\COD$-formula into an equivalent one with constancy dependence atoms only. 

\begin{coro}\label{cod_constancy_nf}
For any \COD-formula $\varphi$, let $\varphi^\ast$ be the formula obtained from $\varphi$ by replacing every occurrence of $\dep{\SET X}{Y}$ by $\bigvee_{\SET x \in Ran(\SET X)} (\SET X = \SET x\wedge  \depc{Y})$. Then $\varphi\dashv\vdash\varphi^\ast$.
\end{coro}

The soundness of $\textsf{DepI}$ and $\textsf{DepE}$ follows from the easily verified fact that
\[\Gamma,\depc{X_1},\dots,\depc{X_n}\models\,\depc{Y}\iff \Gamma\models\,\dep{X_1,\dots,X_n}{Y}.\]
The rule $\textsf{ConI}$ is clearly sound, since $\depc{X}\equiv \bigvvee_{x\in \Ran(X)}X=x$.  The soundness of the more complex rule $\textsf{ConE}$ follows from the equivalence
\begin{equation}\label{eq_phi_phi_f}
    \varphi\equiv\varphi(\bigvvee_{x\in \Ran(X)}X=x/[\depc{X},k])\equiv \bigvvee_{x\in \Ran(X)}\varphi(X=x/[\depc{X},k]),
\end{equation}
which is a special case of a more general equivalence we shall prove in Lemma \ref{instantiation_form} below. In order to state the more general fact, let us first introduce some terminologies.

Let $\mathbf{d}=\langle [\depc{X_1},k_1],\dots,[\depc{X_n},k_n]\rangle$ be a sequence of occurrences of constancy dependence atoms in a formula $\varphi$. A function $f:\{1,\dots,n\}\to \bigcup_{1\leq i\leq n}\Ran(X_i)$ is called an \textbf{instantiating function} over $ \mathbf{d}$ if $f(i)\in \Ran(X_i)$ for each $i$.   
We define an \textbf{$f$-instantiation} of $\varphi$, denoted by $\varphi_{f}$, as the formula 
\[\varphi(X_1=f(1)/\![\depc{X_1},k_1]\,\dots,X_n=f(n)/\![\depc{X_n},k_n]).\]
If $\mathbf{d}$ lists all occurrences of constancy dependence atoms in $\varphi$, then we call the  formula $\varphi_{f}$ a \textbf{full instantiation} of $\varphi$.

\begin{lm}\label{instantiation_form}
 Let $\mathbf{d}$ be a sequence of occurrences of constancy atoms in a $\COD$-formula $\varphi$, and $F$ denote the set of all  instantiating functions over $\mathbf{d}$. Then $\varphi\equiv\bigvvee_{f\in F}\varphi_{f}$. 
\end{lm}
\begin{proof}
We prove the proposition by induction on $\varphi$. If $\varphi$ is a constancy dependence atom $\depc{X_1}$, it is easy to verify that $\depc{X_1}\equiv (\bigvvee_{x\in \Ran(X_i)}X_1=x)=(\bigvvee_{f\in F}X_1=f(1))$.
If $\varphi$ is a $\CO$-formula, then the claim is trivial. 

If $\varphi=\psi\ast\chi$ for $\ast\in\{\wedge,\vee\}$, then the claim follows from the induction hypothesis and the fact that $(\eta\vvee\theta)\ast\chi\equiv (\eta\ast\chi)\vvee(\theta\ast\chi)$. If $\varphi=\SET X=\SET x\cf \psi$, then by induction hypothesis we have that $\psi\equiv\bigvvee_{f\in F}\psi_{f}$. We then have that
\[\SET X=\SET x\cf\psi\,\equiv\, \SET X=\SET x\cf\bigvvee_{f\in F}\psi_{f}\,\equiv\, \bigvvee_{f\in F}( \SET X=\SET x\cf\psi_{f})\,\equiv\, \bigvvee_{f\in F}(\SET X=\SET x\cf\psi)_{f}.\]
\end{proof}

\noindent Now, we know that the equivalence (\ref{eq_phi_phi_f}) holds, and thus our system for  $\COD$ is sound.

For any \COD-formula $\varphi$, consider the equivalent formula $\varphi^\ast$ given by Corollary \ref{cod_constancy_nf}. Denote the set of all full  instantiations of $\varphi^\ast$ by $\I(\varphi)$. Since all dependence atoms occurring in $\varphi^\ast$ are constancy atoms, $\I(\varphi)$ is a finite set of $\CO$-formulas.

\begin{coro}\label{instantiation_nf_sem}
$\varphi\equiv\bigvvee\I(\varphi)$.
\end{coro}
\begin{proof}
By Corollary \ref{cod_constancy_nf} and the soundness of the rules, we have $\varphi\equiv\varphi^\ast$. Then by Lemma \ref{instantiation_form}, $\bigvvee\I(\varphi)\equiv \varphi^\ast\equiv\varphi$.
\end{proof}

As a further corollary, we have that each validity of $\COd$ is ``witnessed'' by a $\Co$ validity, one of its instatiations; or, more generally, we have the following analogous of the disjunction property:


\begin{coro}[Instantiation property]\label{instantiation_prop}
Let $\Delta$ be a set of \CO-formulas, and $\varphi$ a \COD-formula. Then $\Delta\models\varphi$ iff $\Delta\models\varphi_f^\ast$ for some full instantiation $\varphi_f^\ast\in \mathcal{I}(\varphi)$ of $\varphi^\ast$ via an instantiating function $f$.
\end{coro}
\begin{proof}
By Corollary \ref{instantiation_nf_sem} and the disjunction property (Theorem \ref{splitting_prop}).
\end{proof}

To prove the completeness, we will make use of the disjunctive normal form $\varphi\equiv\bigvvee\I(\varphi)$. While the global disjunction $\vvee$ in the equivalence is not in the  
 vocabulary of $\COD$, we now show that in our system $\varphi$ behaves proof-theoretically as $\bigvvee\I(\varphi)$, in the sense that it proves the same consequences and it is derivable from the same assumptions. 

\begin{lm}\label{LEMTRACE1COD}\label{instantiation_2_form}
Let $\Gamma\cup\{\varphi,\psi\}$ be a set of $\COD$-formulas. Let  $\mathbf{d}=\langle  \depc{X_1},\dots,\depc{X_n}\rangle$ be a sequence of occurrences of constancy atoms in $\varphi$.
\begin{enumerate}
\item[(i)] For any instantiating function $f$ over $\mathbf{d}$, we have that $\varphi_{f}\vdash\varphi$.
\item[(ii)] If $\Gamma,\varphi_{f}\vdash\psi$ for all instantiating functions  $f$  over $\mathbf{d}$, then $\Gamma,\varphi\vdash\psi$.
\end{enumerate}
\end{lm}
\begin{proof}
(i) By $\textsf{ConI}$, we have $X_i=f(i)\vdash \depc{X_i}$ for each $1\leq i\leq n$. Then we obtain $\varphi(X_1=f(1)/[\depc{X_1},k_1],\dots,X_n=f(n)/[\depc{X_n},k_n])\vdash \varphi$ by repeatedly applying Lemma \ref{monotonicity_derivations}.

(ii) Let $f_1$ be an arbitrary instantiating function over $\mathbf{d}$. For any arbitrary $x\in \Ran(X_1)$, let $g_x$ be the instantiating function over $\mathbf{d}$ such that $g_x(1)=x$ and $g_x(i)=f_1(i)$ for all $2\leq i\leq n$. By the assumption, we have then  $\Gamma,\varphi_{g_x}\vdash\psi$. Since we can obtain this clause for each $x\in \Ran(X_1)$, 
we conclude by $\textsf{ConE}$ that 
\begin{equation}\label{LEMTRACE1COD_eq1}
\Gamma,\varphi(X_2=f_1(2)/[\depc{X_2},k_2],\dots,X_n=f_1(n)/[\depc{X_n},k_n])\vdash\psi
\end{equation}
for an arbitrary instantiating function $f_1$ over $\mathbf{d}$. 
By the same argument, we can show, by using (\ref{LEMTRACE1COD_eq1}) and $\textsf{ConE}$, that 
\[\Gamma,\varphi(X_3=f_2(3)/[\depc{X_3},k_3],\dots,X_n=f_2(n)/[\depc{X_n},k_n])\vdash\psi\]
for arbitrary instantiating function $f_2$ over $\langle [\depc{X_2},k_2],\dots,[\depc{X_n},k_n]\rangle$. Proceeding in the same way, in the end we obtain that $\Gamma,\varphi\vdash\psi$, as required.
\end{proof}

Now, we can prove the completeness theorem by essentially the same argument as in the proof of Theorem \ref{TEOCOMPLCOU}.

\begin{teo}[Completeness]\label{TEOCOMPCODG}
For any set $\Gamma\cup\{\psi\}$ of $\COD$-formulas, we have that 
\(\Gamma\models^{g} \psi\iff \Gamma \vdash\psi .\)
\end{teo}



\begin{proof}
 We prove the direction ``$\Longrightarrow$". 
As in Theorem \ref{TEOCOMPLCOU}, it suffices to prove it for $\Gamma=\{\varphi\}$.
Suppose $\varphi\models \psi$. By Corollary \ref{instantiation_nf_sem}, we have that
\(
\bigvvee\I(\varphi)  \models \bigvvee\I(\psi).
\)
It then follows that for every $\gamma\in\I(\varphi)$, we have that
\(\gamma \models \bigvvee\I(\psi).
\)
Since each $\gamma$ is a $\CO$-formula, by Theorem \ref{splitting_prop}, there is a \CO-formula $\beta\in \I(\psi)$ such that
\(
\gamma \models \beta.
\)
Now, by applying Theorem \ref{completeness_co} and Lemma \ref{instantiation_2_form}(i), we obtain   that
\(
\gamma \vdash \beta\vdash \psi
\)
for all $\gamma\in \I(\varphi)$, which implies  by Lemma \ref{LEMTRACE1COD}(ii) that $\Gamma \vdash \psi$.
\end{proof}


\subsection{Axiomatizing $\COv$ and $\COd$  over causal teams}\label{AXCOCT}

In this section, we axiomatize $\COv$ and $\COd$  over causal teams. 
It can be verified that all the rules in the systems of the two logics over generalized causal teams are also sound over causal teams; in the case of rules 
 without discharging of assumptions, this immediately follows from the fact that causal teams can be regarded as generalized causal teams with uniform function components. To obtain complete systems over causal teams, we shall then add new axioms or rules to characterize the property of having uniform function components (and thus being ``indistinguishable'' from a causal team). Recall from Corollary \ref{LEMCHARCT} that the property is defined by the formula $\bigvvee_{\mathcal F \in \FUN}\Phi^\mathcal F$, which is, however, not a formula of \COd. To give a unified axiomatization for both \COv~ and \COd, we will instead extend the corresponding system with a new rule. Let us now present our systems. 


\begin{df}
The system for $\COV$ over causal teams consists of all rules of $\COV$ over generalized causal teams (Definition \ref{COV_system_df}) plus the following rule:
\begin{center}{\normalfont
\begin{tabular}{|C{0.96\linewidth}|}
\hline
\AxiomC{}\noLine\UnaryInfC{$\forall \F\in \FUN$}\noLine\UnaryInfC{[$\Phi^\F$]}\noLine\UnaryInfC{$\vdots$} \noLine\UnaryInfC{$ 
\psi$}\RightLabel{\textsf{FunE}}\UnaryInfC{$\psi$}\noLine\UnaryInfC{}\DisplayProof \\\hline
\end{tabular}
}
\end{center}

Similarly, the system for $\COD$ over causal teams consists of all rules of $\COD$ over generalized causal teams (Definition \ref{system_cod_g}) plus the above rule \textsf{FunE}.
\end{df}

The new rule \textsf{FunE} states that if $\psi$ can be derived from $\Phi^\F$ for every $\F\in \FUN$, then $\psi$ can be derived. The soundness of this rule  follows immediately from Corollary \ref{LEM_unf_ct}.
We write $\vdash^c$ for the entailment relation associated with the system of $\COV$ or $\COD$ over  causal teams (which system is meant will always be clear from the context). Obviously, $\vdash^c$ is an extension of the entailment relation $\vdash^g$ for the corresponding system over generalized causal teams.

\begin{lm}\label{FunE_lm}
For any set $\Gamma\cup\{\psi\}$ of $\COV$-formulas (or $\COD$-formulas), we have that   
\[\Gamma,\bigvvee_{\F\in \FUN}\Phi^\F\models^{g}\psi \iff \Gamma,\Phi^\F\models^{g}\psi\text{ for all }\F\in\FUN\iff \Gamma\vdash^{c}\psi.\]
\end{lm}
\begin{proof}
The first ``$\iff$" is easy to verify. For the second ``$\iff$", suppose first that $\Gamma,\Phi^\F\models^{g}\psi$ for all $\F\in\FUN$. Then we have $\Gamma,\Phi^\F\vdash^{g}\psi$ for all $\F\in\FUN$ by the completeness theorem of the \COV-system (or \COD-system) over generalized causal teams. This then gives that $\Gamma,\Phi^\F\vdash^{c}\psi$ for all $\F\in\FUN$. Finally we derive $\Gamma\vdash^c\psi$ by  applying the rule $\textsf{FunE}$. Conversely, if $\Gamma\vdash^{c}\psi$, then we have $\Gamma\models^{c}\psi$ by the soundness theorem. But then we conclude by Lemma \ref{LEMCTMODELS} that $\Gamma,\bigvvee_{\F\in \FUN}\Phi^\F\models^{g}\psi$, which completes the proof.
\end{proof}

\begin{teo}[Completeness]For any set $\Gamma\cup\{\psi\}$ of $\COV$-formulas (or \COD-formulas), we have that  
\(\Gamma\models^{c}\psi \iff \Gamma\vdash^{c}\psi.\)
\end{teo}
\begin{proof}
It suffices to show the direction ``$\Longrightarrow$".
Suppose $\Gamma\models^{c}\psi$. By Lemma \ref{LEMCTMODELS}, we have that $\Gamma,\bigvvee_{\mathcal F\in \FUN} \Phi^{\mathcal F}\models^{g} \psi$, which then gives $\Gamma\vdash^c \psi$ by Lemma \ref{FunE_lm}. 
\end{proof}

Now that we have axiomatized the two logics $\COv$ and $\COd$ over causal teams, let us devote the rest of the section to alternative complete systems for these logics, or equivalently, alternative ways to axiomatize the property of a generalized causal team having a uniform function component.

As mentioned already, such a property can be axiomatized in \COV~ using the following axiom instead of the rule \textsf{FunE}:
{\normalfont
\begin{center}
\begin{tabular}{|C{0.96\linewidth}|}
\hline
\AxiomC{}\noLine\UnaryInfC{}\noLine\UnaryInfC{} \RightLabel{$\textsf{Unf}_{\footnotesize\vvee}$}\UnaryInfC{$\bigvvee_{\mathcal F\in \FUN} \Phi^{\mathcal F}$}\noLine\UnaryInfC{}\DisplayProof\\\hline
\end{tabular}
\end{center}}
\noindent It is straightforward to show that the rule \textsf{FunE} and the above axiom $\textsf{Unf}_{\footnotesize\vvee}$ are inter-derivable from the rules $\vvee\textsf{I}$ and  $\vvee\textsf{E}$. Therefore, the system consisting of all rules of $\COV$ over generalized causal teams together with the axiom $\textsf{Unf}_{\footnotesize\vvee}$ is also sound and complete for \COV~ over causal teams.

We now turn to the logic \COd. Consider the following axiom schema: For $s,t\in \ASS$  and $\F,\G\in \FUN$,
\begin{center}
\begin{tabular}{|C{0.96\linewidth}|}
\hline
\AxiomC{}\noLine\UnaryInfC{}\noLine\UnaryInfC{} \RightLabel{$\textsf{Unf}_{\textsf{D}}$}\UnaryInfC{$ \Xi^{\{(s,\F),(t,\G)\}}$}\noLine\UnaryInfC{}\DisplayProof \\
{\footnotesize where $\F\not\sim\G$\quad\quad\quad\quad}
\\\hline
\end{tabular}
\end{center}
We now show that the above axiom schema $\textsf{Unf}_{\textsf{D}}$ also defines uniformity of generalized causal teams.

\begin{lm}
For any generalized causal team $T$ over some signature $\sigma$,  
\[T\models \Xi^{\{(s,\F),(t,\G)\}}\text{ for all }s,t\in \ASS\text{ and }\F,\G\in \FUN\text{ with }\F\not\sim\G \iff T\text{ is uniform}.\]
\end{lm}
\begin{proof}
Suppose $T$ is uniform. Then $T=T^{\mathcal{H}}$ for some $\mathcal{H}\in \FUN$. For any $\F,\G\in \FUN$ with $\F\not\sim\G$, it must be that either $\F\not\sim\mathcal{H}$ or $\G\not\sim\mathcal{H}$. Thus, for any $s,t\in \ASS$, we have $\{(s,\F),(t,\G)\}\not\preccurlyeq T$, which by Lemma \ref{LEMNOTSUB} implies that $T\models \Xi^{\{(s,\F),(t,\G)\}}$. 

Conversely, suppose $T$ is not uniform. Then, there  exist $(s,\F),(t,\G)\in T$ such that $\F\not\sim\G$. Clearly $\{(s,\F),(t,\G)\}\preccurlyeq T$. But this means, by Lemma \ref{LEMNOTSUB}, that $T\not\models \Xi^{\{(s,\F),(t,\G)\}}$. 
\end{proof}

The above lemma implies (by Corollary \ref{LEMCHARCT}) that the axiom schema $\textsf{Unf}_{\textsf{D}}$ is equivalent to the axiom $\textsf{Unf}_{\footnotesize \vvee}$, and thus to the rule \textsf{FunE}, given  $\vvee\textsf{I}$ and  $\vvee\textsf{E}$. From these it follows that the system consisting of all rules of $\COD$ over generalized causal teams together with the axiom  schema $\textsf{Unf}_{\textsf{D}}$ is sound and complete for \COD~ over causal teams.

\section{Conclusion}\label{sec:conclusion}

\todob{Say something about extending results with disjunctive antecedents, such as in Briggs or Santorio.}

We have answered the main questions concerning the expressive power and the existence of complete deduction calculi for the languages that were proposed in \cite{BarSan2018} and \cite{BarSan2019}, and which involve both interventionist counterfactuals and (contingent) dependencies. 
We have introduced a generalized causal team semantics that extends the causal discourse to cases in which there may be uncertainty about the causal laws; and we have analyzed the properties of the languages from \cite{BarSan2018,BarSan2019} also when they are evaluated in this generalized semantics.
We believe the present work may provide guidelines for the investigation of further notions of dependence and causation in causal team semantics and its variants.

Our work shows that many of the methodologies developed in the literature on team semantics can be adapted to the (generalized) causal team semantics. 
At the same time, a number of peculiarities emerged that set apart these semantic frameworks from the usual team semantics. 
 An interesting anomaly (as noted in \cite{BarSan2019}) is the existence of formulas whose interpretation only depends on the causal structure; this is the case for the  ``direct cause'' formula mentioned in the introduction (call it $\beta_{\mathrm{DC}}(X,V)$), or for the following formula:
 \[
 \beta_{\mathrm{En}}(V):=\bigvee_{X\in\SET W_V} \beta_{\mathrm{DC}}(X,V),
 \]
 which can be shown to characterize the property of $V$ being endogenous (via a non-constant function).
 One unusual aspect of such formulas is that not only the class of causal teams they define is downward closed, but also its complement  (taken together with the empty team) is. 
 Some more peculiarities have emerged that only concern the original causal team semantics: first, the failure of the disjunction property for $\vvee$, and of analogous instantiation properties 
  for the dependence atoms; and secondly, the collapse of $\vvee$ (global disjunction) to $\lor$ (tensor disjunction) for some special classes of disjuncts. This kind of ``collapse'' is possible in pure team semantics (and over generalized causal teams) only if one of the disjuncts entails the other, which is however not the case in some of the examples we isolated. 

 We point out that our calculi are sound only for \emph{recursive} systems, i.e., when the causal graph is acyclic. The general case (and special cases such as the ``Lewisian'' systems considered in \cite{Zha2013}) will require a separate study. We point out, however, that each of our deduction systems can be adapted to the case of \emph{unique-solution} (possibly generalized) causal teams by replacing the \textsf{Recur} rule with an inference rule that expresses the \emph{Reversibility} axiom from \cite{GalPea1998}:
\[
\Large \frac{(\SET{X}=\SET{x} \land Y=y) \cf W=w \hspace{25pt} (\SET{X}=\SET{x} \land W=w)\cf Y=y}{\SET{X}=\SET{x} \cf Y=y}\normalsize \hspace{5pt} \text{ (for $Y\neq W$) }.
\]
See \cite{BarSan2019} for the details on how the semantics should be adapted for this case. The analysis of causal languages beyond the unique-solution case leads naturally to the consideration of operators (such as the \emph{might}-counterfactuals) that do not preserve downward closure. It is difficult to guess, at this stage, to what extent the methodologies of team semantics can be adapted to the study of causal languages that are not downward closed. When this issue is solved, it might be interesting to investigate languages that also include non-downward closed  contingent (in)dependencies, such as the \emph{independence atoms} and \emph{inclusion} atoms considered in the literature on pure team semantics (\cite{GraVaa2013,Gal2012}). \corrb{Finally, it would be interesting to consider counterfactuals with more complex antecedents -- e.g., antecedents containing disjunctions or negations. There have been a number of (often conflicting) proposals in the literature for the semantics of such interventionist counterfactuals \cite{Bri2012,Alo2009,CiaZhaCha2018,Sch2018,San2019}; it is our hope that causal team semantics can shed some light on this issue. For example, a natural interpretation of a counterfactual $(X=1\lor X=2)\cf \psi$ is that $\psi$ certainly holds, even if we are not sure which of the interventions $do(X=1)$ or $do(X=2)$ has been performed. This uncertainty may be modeled by taking the union of the two (generalized) causal teams resulting from applying each intervention.}


\section*{Acknowledgments}
We would like to thank Johan van Benthem for an interesting discussion with the first author, which inspired us to further generalize causal team semantics towards the direction taken in this paper. We also thank two anonymous referees for their helpful comments.


The first author was supported by grant 316460 of the Academy of Finland. 
The second author was supported by grant 330525 of the Academy of Finland, and Research Funds of the University of Helsinki.


\section*{References}

\bibliographystyle{elsarticle-harv}


\end{document}